\documentclass[11pt]{article}
\usepackage{latexsym,a4wide}
\usepackage{amssymb}
\usepackage{amsmath}
\usepackage{color}
\usepackage{graphicx}
\usepackage{amsthm}
\usepackage{indentfirst}

\allowdisplaybreaks%%%

\newcommand \bel {\begin{equation}\label}
\newcommand \ee {\end{equation}}
\newcommand \be {\begin{equation}}

\newcommand \RR {\mathbb R}
\newcommand \HH {\mathbb H}
\newcommand \ZZ {\mathbb Z}
\newcommand \CC {\mathbb C}

\newcommand \LL {\mathbb L}
\newcommand \PP {\mathbb P}
\newcommand \NN {\mathbb N}

\newcommand \del \partial

\newcommand \Ccal {\mathcal C}
\newcommand \Rcal {\mathcal R}
\newcommand \Ocal {\mathcal O}
\newcommand \Lcal {\mathcal L}
\newcommand \Jcal {\mathcal J}

\newcommand \Bcal {\mathcal B}
\newcommand \Fcal {\mathcal F}
\newcommand \bei {\begin{itemize}}
\newcommand \eei {\end{itemize}}

\def \eps {\varepsilon}

\newtheorem{theorem}{\color{black}\indent Theorem}[section]
\newtheorem{lemma}{\color{black}\indent Lemma}[section]
\newtheorem{proposition}{\color{black}\indent Proposition}[section]

\newtheorem{remark}{\color{black}\indent Remark}[section]

\begin{document}
\large
\title{Asymptotic stability of explicite infinite energy blowup solutions for three dimensional incompressible Magnetohydrodynamics equations}
\author{
%\newline
% {\sl Key Words.}
%%
%{\sl Mathematics Subject Classification}.
{\sc Weiping Yan}
\thanks{School of Mathematics, Xiamen University, Xiamen 361000, P.R. China. \textbf{Corresponding author}. Email: yanwp@xmu.edu.cn}
%\thanks{Laboratoire Jacques-Louis Lions, Sorbonne Universit$\acute{e}$, 4, Place Jussieu, 75252 Paris, France.}
%\quad
%{\sc }
%\thanks{}
}
%\date{September 24, 2018}

\maketitle

\begin{abstract}
This paper is denoted to the study of dynamical behavior near explicit finite time blowup solutions for three dimensional incompressible Magnetohydrodynamics (MHD) equations.
More precisely, we find a family of explicit finite time blowup solutions admitted smooth initial data and infinite energy in whole space $\RR^3$. After that,
we prove asymptotic stability of those explicit finite time blowup solutions for $3$D incompressible Magnetohydrodynamics equations in a smooth bounded domain with free surface
$$
\Omega_{t}:=\Big\{(t,x_1,x_2,x_3):0\leq x_i\leq\sqrt{\overline{T}^*-t},\quad t\in(0,\overline{T}^*),\quad i=1,2,3\Big\},
$$
where $\overline{T}^*$ denotes the blowup time.
This means we construct a family of \textbf{stable} blowup solutions for $3$D incompressible Magnetohydrodynamics equations with smooth initial data in $\Omega_t$.
\end{abstract}

\tableofcontents

%=========================================================================

\section{Introduction and main results}
\setcounter{equation}{0}

The incompressible Magnetohydrodynamics equations (MHD) describes the dynamics of electrically conducting fluids arising from plasmas or some other physical phenomena.
In the present paper, we are interested in the stable blowup phenomena of smooth solutions to the three dimensional MHD equations
\bel{E1-1}
\aligned
&\del_t\textbf{v}+\textbf{v}\cdot\nabla \textbf{v}+\nabla P=\nu\triangle \textbf{v}+(\nabla\times\textbf{H})\times\textbf{H},\\
&\del_t\textbf{H}=\mu\triangle\textbf{H}+\nabla\times(\textbf{v}\times\textbf{B}),\\
&\nabla\cdot\textbf{v}=0,\quad \nabla\cdot\textbf{H}=0,
\endaligned
\ee
where $(t,x)\in\RR\times\RR^3$, $\textbf{v}$ denotes the $3$D velocity field of the fluid, $P$ stands for the pressure in the fluid, $\textbf{H}$ is the the magnetic field,
$\nu\geq0$ and $\mu\geq0$ denote the viscosity constant and resistivity constant, respectively.
The divergence free condition in second equations of (\ref{E1-1}) guarantees the incompressibility of the fluid.
In particularly, when $\nu=\mu=0$, equations (\ref{E1-1}) is called ideal incompressible MHD; When $\mu>0$, equations (\ref{E1-1}) is called resistive incompressible MHD.

Assume that $\nu>0$ and $\mu>0$.
It is easy to check that solutions of $3$D incompressible MHD equations (\ref{E1-1}) admits the scaling invariant property, that is,
let $(\textbf{v},\textbf{H},P)$ be a solution of (\ref{E1-1}),  then for any constant $\lambda>0$, the functions
$$
\aligned
&\textbf{v}_{\lambda,\alpha}(t,x)=\lambda\textbf{v}(\lambda^2t,\lambda x),\\
&\textbf{H}_{\lambda,\alpha}(t,x)=\lambda\textbf{H}(\lambda^2t,\lambda x),\\
&P_{\lambda,\alpha}(t,x)=\lambda^2P(\lambda^2t,\lambda x),
\endaligned
$$
are also solutions of $3$D incompressible MHD equations (\ref{E1-1}).
Here the initial data $(\textbf{v}_0(x),\textbf{H}_0(x))$ is changed into $(\lambda\textbf{v}_0(\lambda x),\lambda\textbf{H}_0(\lambda x))$.

The question of finite time singularity/global regularity for three dimensional incomprsssible Navier-Stokes equations is the most important open problems in mathematical fluid mechanics \cite{F}.
Since the three dimensional incomprsssible MHD equations (\ref{E1-1}) is a combination of the Navier-Stokes equations of fluid dynamics and Maxwell's equations of electromagnetism,
it is also natural important problem for the three dimensional incompressible MHD equations.
Toward the well-posedness theory direction, it is natural to expect the global existence of classical solutions for viscous and resistive MHD equations for small initial data \cite{DL,ST}.
More precisely, Sermange and Temam \cite{ST} established the local well-posedness of classical solutions for fully viscous MHD equations,
in which the global well-posedness is also proved in two dimensions.
Lin-Zhang \cite{LZ} proved that  the global well-posedness of a three dimensional incompressible MHD type equations with smooth initial data that is close to some nontrivial steady state.
After that, a simpler proof was offered by Lin-Zhang \cite{LZ1}.
Recently, Abidi-Zhang \cite{AZ} showed the global well-posedness for the three dimensional MHD equations without the admissible restriction in the Lagrangian coordinate system.
The global stability of Alfv\'{e}n waves \cite{A} has been obtained by He-Xu-Yu \cite{HXY} and Cai-Lei \cite{CL}, meanwhile, those results are related to
the vanishing dissipation limit from a fully dissipative MHD system to an inviscid and non-resistive MHD equations.
Wei and Zhang [25] proved the MHD equations with small viscosity and resistivity coefficients
are globally well-posed if the initial velocity is close to $0$ and the initial magnetic field is close to a homogeneous magnetic field in the weighted H\"{o}lder space,
where the closeness is independent of the dissipation coefficients. Pan-Zhou-Zhu \cite{PZZ} gave the global existence of classical solutions to the three dimensional incompressible viscous MHD equations without magnetic diffusion in three dimensional torus. Chemin-McCormick-Robinson-Rodrigo \cite{CMRR} obtained the local existence of solutions to the viscous, non-resistive MHD equations in $\RR^n$ with $n=2,3$. Li-Tan-Yin \cite{LTY} improved the results in \cite{CMRR} in homogeneous Besov spaces.

For large initial data case, there are some numerical results to approach the singularity of this kind problem \cite{GO}. The Beale-Kato-Majda's blowup criterion for incompressible MHD was obtained in \cite{CKS,CM}.
Chae \cite{Ch} excluded the scenario of the apparition of finite time singularity in the form of self-similar singularities.
Very recently, Yan \cite{Yan3} found one family of \textbf{stable} explicit infinite energy blowup solutions for $3$D incompressible Navier-Stokes equations (\ref{E1-1}) with $x\in\RR^3$.
We remark there may be other kind of explicit infinite energy blowup solutions, but most of them are \textbf{unstable}! For example, we take the velocity
$$
\textbf{v}(t,x)={\textbf{c}\over T-t},~~\textbf{c}~denotes~nonzero~constant~ vector,
$$
and the pressure
$$
P(t,x)={x\over T-t}.
$$
One can check above solution is unstable. \textbf{Assume that the blowup $T=1$, then one can also check the function $\textbf{v}=({1\over 1-t},0,0)^T$ is an unstable solution for three dimensional incompressible Navier-Stokes equations.
This means that it is not a genuine infinite energy blowup solution.}

Toward this direction, our first result show there exist a family of explicit infinite energy blowup solutions to incompressible MHD equations (\ref{E1-1}) with smooth initial data
( \cite{Yan0} is a part of this paper).

\begin{theorem}
Let constant $T^*>0$ be maximal existence time and constants $\nu,\mu\geq0$.
The 3D incompressible MHD equations (\ref{E1-1}) admits a family of explicit finite time blowup solutions with smooth initial data as follows
\bel{E1-7R1}
\aligned
&\overline{\textbf{v}}_{T^*}(t,x)=\Big(\overline{v}_1(t,x),\overline{v}_2(t,x),\overline{v}_3(t,x)\Big)^T,\quad (t,x)\in[0,T^*)\times\RR^3,\\
&\overline{\textbf{H}}_{T^*}(t,x)=\Big(\overline{H}_1(t,x),\overline{H}_2(t,x),\overline{H}_3(t,x)\Big)^T,\quad (t,x)\in[0,T^*)\times\RR^3,
\endaligned
\ee
where
$$
\aligned
&\overline{v}_1(t,x):={ax_1\over T^*-t}+kx_2(T^*-t)^{2a},\\
&\overline{v}_2(t,x):={ax_2\over T^*-t}-kx_1(T^*-t)^{2a},\\
&\overline{v}_3(t,x):=-{2ax_3\over T^*-t},
\endaligned
$$
and
$$
\aligned
&\overline{H}_1(t,x):=\bar{a}x_1+{2\bar{a}kx_2x_3(T^*-t)^{2a+1}\over 4a+1},\\
&\overline{H}_2(t,x):=\bar{a}x_2-{2\bar{a}kx_1x_3(T^*-t)^{2a+1}\over 4a+1},\\
&\overline{H}_3(t,x):=-2\bar{a}x_3,
\endaligned
$$
with the pressure
\bel{XX1-2}
\aligned
\overline{P}(t,x)&={x_1^2+x_2^2\over 2}\Big(k^2(T^*-t)^{4a}-{a(a+1)\over (T^*-t)^2}-{8\bar{a}^2k^2x_3^2(T^*-t)^{2(2a+1)}\over (4a+1)^2}\Big)\\
&\quad+x_3^2\Big({a(1-2a)\over (T^*-t)^2}-{2\bar{a}^2k^2r^2(T^*-t)^{2(2a+1)}\over (4a+1)^2}\Big),
\endaligned
\ee
and the smooth initial data
\bel{E1-8R1}
\aligned
&\overline{\textbf{v}}_{T^*}(0,x)=\Big({ax_1\over T^*}+kx_2(T^*)^{2a},~{ax_2\over T^*}-kx_1(T^*)^{2a},~-{2ax_3\over T^*}\Big)^T,\\
&\overline{\textbf{H}}_{T^*}(0,x)=\Big(\bar{a}x_1+{2\bar{a}kx_2x_3(T^*)^{2a+1}\over 4a+1},\bar{a}x_2-{2\bar{a}kx_1x_3(T^*)^{2a+1}\over 4a+1},-2\bar{a}x_3\Big)^T,
\endaligned
\ee
where constants $k,\bar{a}\in\RR/\{0\}$ and $a\in\RR/\{-{1\over4},0\}$.
\end{theorem}

\begin{remark}
It follows from (\ref{E1-7R1}) that
$$
\aligned
&\nabla \overline{v}_1(t,x)=\Big({a\over T^*-t},~k(T^*-t)^{2a},~0\Big)^T,\\
&\nabla \overline{v}_2(t,x)=\Big(-k(T^*-t)^{2a},~{a\over T^*-t},~0\Big)^T,\\
&\nabla \overline{v}_3(t,x)=\Big(0,~0,~-{2a\over T^*-t}\Big)^T,\\
&\nabla \overline{H}_1(t,x)=\Big(\bar{a},{2\bar{a}kx_3(T^*-t)^{2a+1}\over 4a+1},{2\bar{a}kx_2(T^*-t)^{2a+1}\over 4a+1}\Big)^T,\\
&\nabla \overline{H}_2(t,x)=\Big(-{2\bar{a}kx_3(T^*-t)^{2a+1}\over 4a+1},\bar{a},-{2\bar{a}kx_1(T^*-t)^{2a+1}\over 4a+1}\Big)^T,\\
&\nabla \overline{H}_3(t,x)=\Big(0,0,-2\bar{a}\Big)^T,
\endaligned
$$
which means that
$$
div(\overline{v}_i)|_{x=x_0}=\infty,\quad as\quad t\rightarrow (T^*)^-,
$$
and for $a<-{1\over 2}$, there is
$$
div(\overline{H}_i)|_{x=x_0}=\infty,\quad as\quad t\rightarrow (T^*)^-,
$$
for a fixed point $x_0\in\RR^3$. Here one can see the initial data is smooth from (\ref{E1-8R1}). But the initial data goes to infinity as $x\rightarrow\infty$.

\end{remark}

On one hand, it is easy to see the blowup phenomenon of $3$D incompressible MHD (\ref{E1-1}) can only take place in the velocity field of the fluid $\textbf{v}$, but no blowup for the magnetic field $\textbf{H}$. Moreover, the blowup solutions (\ref{E1-7R1}) independent of viscosity constant $\nu$ and resistivity constant $\mu$, so our results also hold for both $3$D ideal incompressible MHD and resistive incompressible MHD. On the other hand, if the magnetic field $\textbf{H}\equiv0$, equations (\ref{E1-1}) is reduced into $3$D incompressible Navier-Stokes equations. Then corresponding explicit blowup solutions given in (\ref{E1-7R1}) are also explicit stable blowup solutions for $3$D incompressible Navier-Stokes equations \cite{Yan3}. For the velocity field of the fluid $\textbf{v}$, it follows from (\ref{E1-7R1}) that there is self-smilar singularity in $x_3$ direction, that is, $-{2ax_3\over T^*-t}$ for $a\in\RR/\{0\}$.
Moreover, by (\ref{E1-7R1}), there are not only blowup for velocity field of the fluid $\textbf{v}$, but also blowup for the magnetic field $\textbf{H}$ with constant $a<-{1\over 2}$ as $t\rightarrow (T^*)^-$.

Let the smooth bounded domain be the form
\bel{B1}
\Omega_{t}:=\Big\{(t,x_1,x_2,x_3):0\leq x_i\leq\sqrt{\overline{T}^*-t},\quad t\in(0,\overline{T}^*),\quad i=1,2,3\Big\},
\ee
which is a free boundary surface. The second result is devoted to the study of nonlinear stable of singular solutions (\ref{E1-7R1}) in this smooth bounded domain.
We set
$$
\aligned
&\textbf{v}(t,x)=\textbf{w}(t,x)+\overline{\textbf{v}}_{\overline{T}^*}(t,x),\\
&\textbf{H}(t,x)=\textbf{b}(t,x)+\overline{\textbf{H}}_{\overline{T}^*}(t,x),\\
&P(t,x)=p(t,x)+\overline{P}(t,x),
\endaligned
$$
then substituting above equalities into the three dimensional MHD system (\ref{E1-1}) to get the perturbation system as follows
$$
\aligned
\textbf{w}_t-\nu\triangle\textbf{w}&=\nabla p-\textbf{w}\cdot\nabla\overline{\textbf{v}}_{\overline{T}^*}-(\overline{\textbf{v}}_{\overline{T}^*}+\textbf{w})\cdot\nabla\textbf{w}
+(\overline{\textbf{H}}_{\overline{T}^*}+\textbf{b})\cdot\nabla\textbf{b}\\
&\quad\quad+\textbf{b}\cdot\nabla\overline{\textbf{H}}_{\overline{T}^*}-\nabla(\overline{\textbf{H}}_{\overline{T}^*}\cdot\textbf{b})-\nabla({|\textbf{b}|^2\over 2}),\\
\textbf{b}_t-\mu\triangle\textbf{b}&=(\overline{\textbf{H}}_{\overline{T}^*}+\textbf{b})\cdot\nabla\textbf{w}+\textbf{b}\cdot\nabla\overline{\textbf{v}}_{\overline{T}^*}
-\overline{\textbf{v}}_{\overline{T}^*}\cdot\nabla\textbf{b}-\textbf{w}\cdot\nabla(\overline{\textbf{H}}_{\overline{T}^*}+\textbf{b}),\\
&\nabla\cdot\textbf{w}=0,\quad \nabla\cdot\textbf{b}=0.
\endaligned
$$

Obviously, there are singular coefficients like ${1\over\overline{T}^*-t}$ in above perturbation system. It causes large difficulty to solve it directly. 
In order to overcome this case, we introduce the self-similarity coordinates
\bel{AEA1-1}
\aligned
&\tau=-\ln(\overline{T}^*-t)+\ln \overline{T}^*,\\
&y={x\over \sqrt{\overline{T}^*-t}},
\endaligned
\ee
where one can see the blowup time $\overline{T}^*>0$ has been transformed into $+\infty$.
Thus the smooth bounded domain (\ref{B1}) is transformed into a fixed domain
$$
\overline{\Omega}:=\{(\tau,y):0<\tau<+\infty,\quad y\in\Omega:=([0,1])^3\}.
$$
So the local existence of perturbation system is equivalent to prove the global existence of perturbation system. 

In fact, this kind of domain has been widely encountered when one studied the stabliliy of self-similar blowup solutions for wave equations (e.g. see \cite{D1,D2}).
The main reason is the propagation of singularity inside the light cone for wave equations. Since the explicit blowup solutions given in (\ref{E1-7R1}) can also propagate inside the light cone,
we study nonlinear stability of blowup solutions (\ref{E1-7R1}) in the free boundary surface (\ref{B1}).

We supplement the MHD system (\ref{E1-1}) with initial data
$$
\textbf{v}(0,x)=\textbf{v}_0(x),\quad \textbf{H}(0,x)=\textbf{H}_0(x).
$$
and boundary condition
\bel{Ay1-1}
\aligned
&\Big(\textbf{v}(t,x)-\overline{\textbf{v}}_{\overline{T}^*}(t,x)\Big)|_{x\in\del\Omega_t}=\textbf{w}(t,x)|_{x\in\del\Omega_t}=0,\\
&\Big(\textbf{H}(t,x)-\overline{\textbf{H}}_{\overline{T}^*}(t,x)\Big)|_{x\in\del\Omega_t}=\textbf{b}(t,x)|_{x\in\del\Omega_t}=0.
\endaligned
\ee

We now state the asymptotic stability of infinite energy blowup solutions (\ref{E1-7R1}).

\begin{theorem}
Let viscosity constant $\nu$ and resistivity constant $\mu$ be sufficient big, constants $a\in(0,{1\over2}]$, $\bar{a},k\in(0,1]$, a fixed integer $s\geq2$.
The family of explicit finite time blowup solutions (\ref{E1-7R1}) is asymptotic stability in $\Omega_t$, i.e.
for a sufficient small $\eps>0$, if
$$
\|\textbf{v}_0(x)-\overline{\textbf{v}}_{T^*}(0,x)\|_{H^s(\Omega_0)}+\|\textbf{H}_0(x)-\overline{\textbf{H}}_{T^*}(0,x)\|_{H^s(\Omega_0)}<\eps,
$$
then the three dimensional incompressible MHD equations (\ref{E1-1}) admits a local solution $(\textbf{v}(t,x),\textbf{H}(t,x))$ such that
$$
\aligned
&\textbf{v}(t,x)=\overline{\textbf{v}}_{\overline{T}^*}(t,x)+\textbf{w}(t,x),\\
&\textbf{H}(t,x)=\overline{\textbf{H}}_{\overline{T}^*}(t,x)+\textbf{b}(t,x),
\endaligned
$$
with
$$
\aligned
\|\textbf{w}(t,x)\|_{H^s(\Omega_t)}^2+\|\textbf{b}(t,x)\|_{H^s(\Omega_t)}^2
\lesssim(\overline{T}^*-t)^{C_{\eps,a,\bar{a},k,\nu,\mu}},
\qquad \forall (t,x)\in(0,\overline{T}^*)\times\Omega_t,
\endaligned
$$
with the boundary condition
$$
\textbf{w}(t,x)|_{\del\Omega_t}=0,\quad\textbf{b}(t,x)|_{\del\Omega_t}=0,
$$
where $C_{\eps,a,\bar{a},k,\nu,\mu}$ is a positive constant depending on constants $\eps,a,\bar{a},k,\nu,\mu$, and
$H^s(\Omega_t)$ denotes the usual Sobolev space.

Moreover, the blowup time $\overline{T}^*$ is contained in $[T^*-\delta,T^*+\delta]$ for a positive constant $\delta\ll1$.
\end{theorem}

\begin{remark}
Above stability result also tells us if we perturbe the initial data $(\overline{\textbf{v}}_{T^*}(0,x),\overline{\textbf{H}}_{T^*}(0,x))^T$, then we can construct blowup solutions 
$$
\aligned
&\textbf{v}(t,x)=\overline{\textbf{v}}_{\overline{T}^*}(t,x)+\Ocal(\eps),\\
&\textbf{H}(t,x)=\overline{\textbf{H}}_{\overline{T}^*}(t,x)+\Ocal(\eps),
\endaligned
$$ 
but with the blowup time $\overline{T}^*$ contained in the interval $[T^*-\delta,T^*+\delta]$. So the blowup time maybe shift.
A similar phenomenon has been proven in other kind of evolution equations ( for example, nonlinear wave equation \cite{D1}).
\end{remark}

\begin{remark}
For the three dimensional incompressible Navier-Stokes equations, we notice that the pressure $P$ is uniquely determined by the formula
$$
P(t,x)=-\triangle^{-1}\sum_{i,j=1}^3{\del v_i\over\del x_j}{\del v_j\over\del x_i}.
$$
Hence the nonlinear term $\textbf{v}\cdot\nabla \textbf{v}$ is important for getting the pressure.

In fact, if we consider a simple model
$$
\aligned
&\textbf{v}_t+\nabla P=0,\\
&\nabla\cdot\textbf{v}=0,
\endaligned
$$
then we take both sides with divergence free condition to the equation, thus we can not get any information on the pressure $P$ with the velocity field $\textbf{v}$. 
So this means that the pressure can not be unique determined.
\end{remark}

%In fact, one can easy check that explicit blowup solutions (\ref{E1-7R1}) also solve non-resistive incompressible MHD
%$$
%\aligned
%&\del_t\textbf{v}+\textbf{v}\cdot\nabla \textbf{v}+\nabla P=\nu\triangle \textbf{v}+(\nabla\times\textbf{H})\times\textbf{H},\\
%&\del_t\textbf{H}=\nabla\times(\textbf{v}\times\textbf{B}),\\
%&\nabla\cdot\textbf{v}=0,\quad \nabla\cdot\textbf{H}=0,
%\endaligned
%$$

\paragraph{Notations.}

Thoughout this paper, we denote the usual norm of $\LL^2(\Omega)$ and Sobolev space $\HH^s(\Omega)$ by $\|\cdot\|_{\LL^2}$ and $\|\cdot\|_{\HH^s}$, respectively.
The norm of $L^2$ space $L^2(\Omega):=(\LL^2(\Omega))^3$ and Sobolev space $H^s(\Omega):=(\HH^s(\Omega))^3$ are denoted by $\|\cdot\|_{L^2}$ and  $\|\cdot\|_{H^s}$, repestively. The symbol $a\lesssim b$ means that there exists a positive constant $C$ such that $a\leq Cb$. $(a,b,c)^T$ denotes the column vector in $\Omega$.
The space $\LL^2((0,\overline{T}^*);H^s(\Omega))$ is equipped with the norm
$$
\|u\|^2_{\LL^2((0,\overline{T}^*);H^s(\Omega))}:=\int_0^{\overline{T}^*}\|u(t,\cdot)\|^2_{H^s}dt.
$$
We also introduce the function space $\Ccal^s_{1}:=\bigcap_{i= 0}^1\CC^i((0,\overline{T}^*);H^{s-i}(\Omega))$ with the norm
$$
\|u\|^2_{\Ccal^s_{1}}:=\sup_{t\in(0,\overline{T}^*)}\sum_{i= 0}^1\|\partial^{i}_{t}u\|^2_{H^{s-i}}.
$$
The letter $C$ with subscripts to denote dependencies stands for a positive constant that might change its value at each occurrence.

The organization of this paper is as follows. In section 2, we give the details of finding explicit finite time blowup solutions of $3$D incompressible MHD equations (\ref{E1-1}).
In section 3, we study the local well-posedness for the linearized $3$D incompressible MHD equations (\ref{E1-1}) around explicit finite time blowup solutions with small initial data.
This last section will prove asymptotic stability of those finite time blowup solutions by construction of Nash-Moser iteration scheme.

%--------------------------------------------------------------------------------------------------------

\section{Explicit finite time blowup solutions with infinite energy}
\setcounter{equation}{0}

In this section, we show how to find a family of explicit finite time blowup solutions of 3D incompressible MHD equations (\ref{E1-1}), which contains the result given in \cite{Yan0}.
We first recall a result on the existence of explicit blowup axisymmetric solutions for $3$D incompressible Navier-Stokes equations \cite{Yan3}.
Let $\textbf{e}_r$, $\textbf{e}_{\theta}$ and $\textbf{e}_z$ be the cylindrical coordinate
system,
\bel{E1-0}
\aligned
&\textbf{e}_r=({x_1\over r},{x_2\over r},0)^T,\\
&\textbf{e}_{\theta}=({x_2\over r},-{x_1\over r},0)^T,\\
&\textbf{e}_z=(0,0,1)^T,
\endaligned
\ee
where $r=\sqrt{x_1^2+x_2^2}$ and $z=x_3$.

The 3D incompressible Navier-Stokes equations admits a family of explicit blowup axisymmetric solutions:
\bel{E1-6r1}
\textbf{v}(t,x)=v^r(t,r,z)\textbf{e}_r+v^{\theta}(t,r,z)\textbf{e}_{\theta}+v^z(t,r,z)\textbf{e}_z,\quad (t,x)\in[0,T^*)\times\RR^3,
\ee
where
$$
\aligned
&v^r(t,r,z)={ar\over T^*-t},\\
&v^{\theta}(t,r,z)= kr(T^*-t)^{2a},\\
&v^z(t,r,z)=-{2az\over T^*-t},
\endaligned
$$
where constants $a,k\in\RR/\{0\}$.

We now derive the $3$D incompressible MHD equations (\ref{E1-1}) with axisymmetric velocity field in the cylindrical coordinate (e.g. see \cite{Lei}).
The $3$D velocity field $\textbf{v}(t,x)$ and magnetic field $\textbf{H}(t,x)$ are called axisymmetric if they can be written as
$$
\aligned
&\textbf{v}(t,x)=v^r(t,r,z)\textbf{e}_r+v^{\theta}(t,r,z)\textbf{e}_{\theta}+v^z(t,r,z)\textbf{e}_z,\\
&\textbf{H}(t,x)=H^r(t,r,z)\textbf{e}_r+H^{\theta}(t,r,z)\textbf{e}_{\theta}+H^z(t,r,z)\textbf{e}_z,\\
&P(t,x)=P(t,r,z),
\endaligned
$$
where $(v^r,v^{\theta},v^z)$, $(H^r,H^{\theta},H^z)$ and $P(t,r,z)$ do not depend on the $\theta$ coordinate.

Note that the Lorentz force term
$$
(\nabla\times\textbf{H})\times\textbf{H}=\textbf{H}\cdot\nabla\textbf{H}-\nabla{|\textbf{H}|^2\over 2}.
$$
Then $3$D MHD equations (\ref{E1-1}) with axisymmetric velocity field in the cylindrical coordinates can be reduced into a system as follows
\bel{X1-2}
\del_tv^r+v^r\del_rv^r+v^z\del_zv^r-{1\over r}(v^{\theta})^2+\del_r\bar{P}
=\nu(\triangle-{1\over r^2})v^r+H^r\del_rH^r+H^z\del_zH^r-{1\over r}(H^{\theta})^2,
\ee
\bel{X1-2R0}
\del_tv^{\theta}+v^r\del_rv^{\theta}+v^z\del_zv^{\theta}+{1\over r}v^rv^{\theta}
=\nu(\triangle-{1\over r^2})v^{\theta}+H^r\del_rH^{\theta}+H^z\del_zH^{\theta}+{1\over r}H^{\theta}H^r,
\ee
\bel{X1-2R1}
\del_tv^z+v^r\del_rv^z+v^z\del_zv^z+\del_z\bar{P}=\nu\triangle v^z+H^r\del_rH^z+H^z\del_zH^z,
\ee
\bel{X1-2R2}
\del_tH^r+v^r\del_rH^r+v^z\del_zH^r=\mu(\triangle-{1\over r^2})H^r+H^r\del_rv^r+H^z\del_zv^r,
\ee
\bel{X1-2R3}
\del_tH^{\theta}+v^r\del_rH^{\theta}+v^z\del_zH^{\theta}+{1\over r}H^rv^{\theta}
=\mu(\triangle-{1\over r^2})H^{\theta}+H^r\del_rv^{\theta}+H^z\del_zv^{\theta}+{1\over r}v^rH^{\theta},
\ee
\bel{X1-2R4}
\del_tH^z+v^r\del_rH^z+v^z\del_zH^z=\mu\triangle H^z+H^r\del_rv^z+H^z\del_zv^z,
\ee
where the pressure is given by
\bel{X1-3}
\bar{P}=P+{|\textbf{H}|^2\over 2}.
\ee

The incompressibility condition becomes
\bel{X1-4}
\aligned
&\del_r(rv^r)+\del_z(rv^z)=0,\\
&\del_r(rH^r)+\del_z(rH^z)=0.
\endaligned
\ee

The following result gives a family of explict self-similar blowup solutions for system (\ref{X1-2})-(\ref{X1-3}) with the  incompressibility condition (\ref{X1-4}).

\begin{proposition}
\begin{itemize}
Let $T^*>0$ be a constant. System (\ref{X1-2})-(\ref{X1-3}) with the incompressibility condition (\ref{X1-4}) admits a family of explicit blowup solutions:
\bel{X1-6}
\aligned
&v^r(t,r,z)={ar\over T^*-t},\\
&v^{\theta}(t,r,z)=kr(T^*-t)^{2a},\\
&v^z(t,r,z)=-{2az\over T^*-t},\\
&H^r(t,r,z)=\bar{a}r,\\
&H^{\theta}(t,r,z)={2\bar{a}krz(T^*-t)^{2a+1}\over 4a+1},\\
&H^z(t,r,z)=-2\bar{a}z,
\endaligned
\ee
where constants $\bar{a},k\in\RR/\{0\}$ and $a\in\RR/\{-{1\over 4},0\}$.

\end{itemize}
\end{proposition}
\begin{proof}
The idea of finding explicit blowup solutions for system (\ref{X1-2})-(\ref{X1-3}) with the incompressibility condition (\ref{X1-4}) comes from \cite{Yan3}. This is based on the observation on the structure of system (\ref{X1-2})-(\ref{X1-3}) and incompressibility condition (\ref{X1-4}). We notice that if the magnetic field $\textbf{H}=0$, equations (\ref{E1-1}) is reduced into $3$D incompressible Navier-Stokes equations, so the explicit blowup solutions (\ref{E1-6r1}) of Navier-Stokes equations should be a part of solutions for the corresponding MHD equations.

We set
\bel{X1-6r1}
\aligned
&v^r(t,r,z)={ar\over T^*-t},\\
&v^{\theta}(t,r,z)=kr(T^*-t)^{2a},\\
&v^z(t,r,z)=-{2az\over T^*-t},
\endaligned
\ee
be a part of solutions for (\ref{X1-2})-(\ref{X1-2R4}), where constants $a,k\in\RR/\{0\}$.

Substituting (\ref{X1-6r1}) into equations (\ref{X1-2R0}) and (\ref{X1-2R2})-(\ref{X1-2R4}), we get
\bel{X1-2rr0}
H^r\del_rH^{\theta}+H^z\del_zH^{\theta}+{1\over r}H^{\theta}H^r=0,
\ee
\bel{X1-2rr2}
\del_tH^r+{ar\over T^*-t}\del_rH^r-{2az\over T^*-t}\del_zH^r=\mu(\triangle-{1\over r^2})H^r+{a\over T^*-t}H^r,
\ee
\bel{X1-2rr3}
\del_tH^{\theta}+{ar\over T^*-t}\del_rH^{\theta}-{2az\over T^*-t}\del_zH^{\theta}+2k(T^*-t)^{2a}H^r
=\mu(\triangle-{1\over r^2})H^{\theta}+{a\over T^*-t}H^{\theta},
\ee
\bel{X1-2rr4}
\del_tH^z+{ar\over T^*-t}\del_rH^z-{2az\over T^*-t}\del_zH^z=\mu\triangle H^z-{2a\over T^*-t}H^z,
\ee

We observe the the structure of incompressibility condition (\ref{X1-4}) on the magnetic field, and we find it is better to set
\bel{E2-1}
\aligned
&H^r(t,r,z)={\bar{a}r\over (T^*-t)^{\alpha}},\\
&H^z(t,r,z)=-{2\bar{a}z\over (T^*-t)^{\alpha}},
\endaligned
\ee
where $\bar{a}\neq0,\alpha$ are two unknown constants.

It is easy to see $H^r(t,r,z)$ and $H^z(t,r,z)$ given in (\ref{E2-1}) satisfies the incompressibility condition (\ref{X1-4}) on the magnetic field.

Note that $(\triangle-{1\over r^2})r=0$.
Substituting the $H^r$ in (\ref{E2-1}) into (\ref{X1-2rr2}), we get
$$
\alpha=0.
$$
which gives that
\bel{E2-2}
\aligned
&H^r(t,r,z)=\bar{a}r,\\
&H^z(t,r,z)=-2\bar{a}z,
\endaligned
\ee

We now find $H^{\theta}(t,r,z)$. Assume that
\bel{E2-3}
H^{\theta}(t,r,z)={\bar{k}r^pz^q\over (T^*-t)^{\beta}},
\ee
where $\bar{k}\neq0$ and $p,q,\beta$ are unknown constants.

It is easy to check that $H^r(t,r,z), H^z(t,r,z), H^{\theta}(t,r,z)$ given in (\ref{E2-2})-(\ref{E2-3}) satisfy (\ref{X1-2rr2}) and (\ref{X1-2rr4}) with
\bel{E2-4r0}
p-2q+1=0.
\ee

We substitute (\ref{E2-2})-(\ref{E2-3}) into (\ref{X1-2rr3}), it holds
\bel{E2-5r0}
{\beta\bar{k}r^pz^q\over (T^*-t)^{\beta+1}}+{ap\bar{k}r^pz^q\over (T^*-t)^{\beta+1}}-{2a\bar{k}qr^pz^q\over (T^*-t)^{\beta+1}}+2k\bar{a}r(T^*-t)^{2a}={a\bar{k}r^pz^q\over (T^*-t)^{\beta+1}},
\ee
which gives that
$$
\alpha=-2a-1,\quad p=1,\quad q=1,
$$
and
$$
\bar{k}={2\bar{a}k\over 4a+1}.
$$
Thus we get
\bel{E2-6r0}
H^{\theta}(t,r,z)={2\bar{a}krz(T^*-t)^{2a+1}\over 4a+1}.
\ee

In conclusion, by (\ref{E2-2}) and (\ref{E2-6r0}), we obtain
$$
\aligned
&H^r(t,r,z)=\bar{a}r,\\
&H^{\theta}(t,r,z)={2\bar{a}krz(T^*-t)^{2a+1}\over 4a+1},\\
&H^z(t,r,z)=-2\bar{a}z,
\endaligned
$$
which combining with (\ref{X1-6}) gives a family of solutions of system (\ref{X1-2})-(\ref{X1-3}) with the incompressibility condition (\ref{X1-4}). Here constants $\bar{a},k\in\RR/\{0\}$ and $a\in\RR/\{-{1\over 4},0\}$.

Furthermore, we compute the pressure $P$. We substitute (\ref{X1-6}) into (\ref{X1-2}) and (\ref{X1-2R1}), there are
$$
\del_r\bar{P}=\Big(\bar{a}^2+k^2(T^*-t)^{4a}-{a(1+a)\over (T^*-t)^2}-{4\bar{a}^2k^2z^2(T^*-t)^{2(2a+1)}\over (4a+1)^2}\Big)r,
$$
and
$$
\del_z\bar{P}=2z\Big(2\bar{a}^2+{a(1-2a)\over (T^*-t)^2}\Big).
$$
Note that
$$
|\textbf{H}|^2=\bar{a}^2r^2+{4k^2\bar{a}^2r^2z^2(T^*-t)^{2(2a+1)}\over (4a+1)^2}+4\bar{a}^2z^2.
$$
Thus by (\ref{X1-3}), direct computations give the pressure
$$
\aligned
P(t,r,z)&={r^2\over 2}\Big(k^2(T^*-t)^{4a}-{a(a+1)\over (T^*-t)^2}-{8\bar{a}^2k^2z^2(T^*-t)^{2(2a+1)}\over (4a+1)^2}\Big)\\
&\quad+z^2\Big({a(1-2a)\over (T^*-t)^2}-{2\bar{a}^2k^2r^2(T^*-t)^{2(2a+1)}\over (4a+1)^2}\Big).
\endaligned
$$

\end{proof}

Since
$$
\aligned
&\textbf{v}(t,x)=v^r(t,r,z)\textbf{e}_r+v^{\theta}(t,r,z)\textbf{e}_{\theta}+v^z(t,r,z)\textbf{e}_z,\\
&\textbf{H}(t,x)=H^r(t,r,z)\textbf{e}_r+H^{\theta}(t,r,z)\textbf{e}_{\theta}+H^z(t,r,z)\textbf{e}_z,
\endaligned
$$
we can obtain a family of explicit blowup axisymmetric solutions for $3$D incompressible MHD equations by noticing that $\textbf{e}_r,\textbf{e}_{\theta},\textbf{e}_z$ are defined in (\ref{E1-0}), $r=\sqrt{x_1^2+x_2^2}$ and $z=x_3$, and
$$
\aligned
&v^r(t,r,z)={ar\over T^*-t},\\
&v^{\theta}(t,r,z)=kr(T^*-t)^{2a},\\
&v^z(t,r,z)=-{2az\over T^*-t},\\
&H^r(t,r,z)=\bar{a}r,\\
&H^{\theta}(t,r,z)={2\bar{a}krz(T^*-t)^{2a+1}\over 4a+1},\\
&H^z(t,r,z)=-2\bar{a}z,
\endaligned
$$
where constants $\bar{a},k\in\RR/\{0\}$ and $a\in\RR/\{-{1\over 4},0\}$.

Futhermore the vorticity vector $\omega$ is
$$
\omega(t,x)=\omega^r(t,r,z)\textbf{e}_r+\omega^{\theta}(t,r,z)\textbf{e}_{\theta}+\omega^z(t,r,z)\textbf{e}_z,
$$
where
$$
\aligned
&\omega^r(t,r,z)=-\del_zv^{\theta}=0,\\
&\omega^{\theta}(t,r,z)=\del_zv^r-\del_rv^z=0,\\
&\omega^z(t,r,z)={1\over r}\del_r(rv^{\theta})=2k(T^*-t)^{2a}.
\endaligned
$$

By directly computations, we can obtain a family of explicit blowup solutions from (\ref{E1-6r1}).
Moreover, the vorticity vector
$$
\omega(t,x)=\nabla\times\textbf{v}=2k(T^*-t)^{2a}.
$$

%--------------------------------------------------------------------------------------

\section{Well-posedness for the linearized time evolution}
\setcounter{equation}{0}

Since the dynamical behavior near blowup solutions is the study of local behavior near blowup point, we consider nonlinear stability of explicit blowup solutions in a smooth bounded domain with free boundary as follows
$$
\Omega_{t}:=\Big\{(t,x_1,x_2,x_3):0\leq x_i\leq\sqrt{\overline{T}^*-t},\quad t\in(0,\overline{T}^*),\quad i=1,2,3\Big\}.
$$ 
This section is devoted to the study of the well-posedness of linearized equations.
We take the divergence of the first equation in incompressible MHD (\ref{E1-1}) to get the pressure
$$
\triangle P=\sum_{i,j=1}^3\Big(-{\del v_i\over\del x_j}{\del v_j\over \del x_i}+{\del H_i\over\del x_j}{\del H_j\over \del x_i}\Big)-\triangle({|\textbf{H}|^2\over 2}).
$$

Let
$$
\aligned
&\textbf{v}(t,x)=\textbf{w}(t,x)+\overline{\textbf{v}}_{\overline{T}^*}(t,x),\\
&\textbf{H}(t,x)=\textbf{b}(t,x)+\overline{\textbf{H}}_{\overline{T}^*}(t,x),\\
&P(t,x)=p(t,x)+\overline{P}(t,x),
\endaligned
$$
then substituting above equalities into incompressible MHD equations (\ref{E1-1}), we get the perturbation equations
\bel{E3-1}
\aligned
\textbf{w}_t-\nu\triangle\textbf{w}&=\nabla p-\textbf{w}\cdot\nabla\overline{\textbf{v}}_{\overline{T}^*}-(\overline{\textbf{v}}_{\overline{T}^*}+\textbf{w})\cdot\nabla\textbf{w}
+(\overline{\textbf{H}}_{\overline{T}^*}+\textbf{b})\cdot\nabla\textbf{b}\\
&\quad\quad+\textbf{b}\cdot\nabla\overline{\textbf{H}}_{\overline{T}^*}-\nabla(\overline{\textbf{H}}_{\overline{T}^*}\cdot\textbf{b})-\nabla({|\textbf{b}|^2\over 2}),\\
\textbf{b}_t-\mu\triangle\textbf{b}&=(\overline{\textbf{H}}_{\overline{T}^*}+\textbf{b})\cdot\nabla\textbf{w}+\textbf{b}\cdot\nabla\overline{\textbf{v}}_{\overline{T}^*}
-\overline{\textbf{v}}_{\overline{T}^*}\cdot\nabla\textbf{b}-\textbf{w}\cdot\nabla(\overline{\textbf{H}}_{\overline{T}^*}+\textbf{b}),\\
&\nabla\cdot\textbf{w}=0,\quad \nabla\cdot\textbf{b}=0,
\endaligned
\ee
with initial data
$$
\aligned
&\textbf{w}(0,x)=\textbf{w}_0(x):=\textbf{v}_0(x)-\overline{\textbf{v}}_{T^*}(0,x),\\
&\textbf{b}(0,x)=\textbf{b}_0(x):=\textbf{b}_0(x)-\overline{\textbf{H}}_{T^*}(0,x),
\endaligned
$$
and boundary condition
$$
\textbf{w}(t,x)|_{x\in\del\Omega_t}=0,\quad \textbf{b}(t,x)|_{x\in\Omega_t}=0,
$$
where $(t,x)\in(0,\overline{T}^*)\times\Omega_t$, $\textbf{v}=(v_1,v_2,v_3)^{T}\in\RR^3$, $\textbf{b}=(b_1,b_2,b_3)^T\in\RR^3$
and
\bel{xx1-1}
\nabla\overline{\textbf{v}}=\left(
\begin{array}{ccc}
{a\over \overline{T}^*-t}&-k(\overline{T}^*-t)^{2a}&0\\
k(\overline{T}^*-t)^{2a}&{a\over \overline{T}^*-t}&0\\
0&0&-{2a\over \overline{T}^*-t}
\end{array}
\right),
\ee
\bel{xx1-2}
\nabla\overline{\textbf{H}}=\left(
\begin{array}{ccc}
\bar{a}&{2\bar{a}kx_3(\overline{T}^*-t)^{2a+1}\over 4a+1}&{2\bar{a}kx_2(\overline{T}^*-t)^{2a+1}\over 4a+1}\\
-{2\bar{a}kx_3(\overline{T}^*-t)^{2a+1}\over 4a+1}&\bar{a}&-{2\bar{a}kx_1(\overline{T}^*-t)^{2a+1}\over 4a+1}\\
0&0&-2\bar{a}
\end{array}
\right),
\ee
and the pressure
\bel{E3-0R1}
\aligned
p(t,x)&=-\triangle^{-1}\Big(\sum_{k=1}^3({\del w_k\over\del x_k})^2+2k(\overline{T}^*-t)^{2a}({\del w_2\over \del x_1}-{\del w_1\over \del x_2})
+2{\del w_1\over \del x_2}{\del w_2\over\del x_1}+2{\del w_1\over \del x_3}{\del w_3\over\del x_1}+2{\del w_2\over \del x_3}{\del w_3\over\del x_2}\Big)\\
&\quad+\triangle^{-1}\Big(\sum_{k=1}^3({\del b_k\over\del x_k})^2+{4\bar{a}kx_3(\overline{T}^*-t)^{2a+1}\over 4a+1}({\del b_2\over \del x_1}-{\del b_1\over \del x_2})
+{4\bar{a}kx_2(\overline{T}^*-t)^{2a+1}\over 4a+1}{\del b_3\over\del x_1}\\
&\quad-{4\bar{a}kx_1(\overline{T}^*-t)^{2a+1}\over 4a+1}{\del b_3\over \del x_2}+2\bar{a}({\del b_1\over \del x_1}+{\del b_2\over \del x_2}-2{\del b_3\over \del x_3})
+2{\del b_1\over \del x_2}{\del b_2\over\del x_1}+2{\del b_1\over \del x_3}{\del b_3\over\del x_1}+2{\del b_2\over \del x_3}{\del b_3\over\del x_2}\Big)\\
&\quad-{1\over 2}(b_1^2+b_2^2+b_3^2)-\Big(\bar{a}x_1+{2\bar{a}kx_2x_3(\overline{T}^*-t)^{2a+1}\over 4a+1}\Big)b_1-\Big(\bar{a}x_2-{2\bar{a}kx_1x_3(\overline{T}^*-t)^{2a+1}\over 4a+1}\Big)b_2\\
&\quad+2\bar{a}x_3b_3.
\endaligned
\ee

Let $R\in(0,1)$ be a fixed constant. We define
\bel{E3-2}
\Bcal_{R}:=\{(\textbf{w},\textbf{b}): \|\textbf{w}\|_{\Ccal_1^{s+3}}+\|\textbf{b}\|_{\Ccal_1^{s+3}}\leq R<1,\quad \forall s\in\NN^+\}.
\ee

Assume that fixed functions $(\textbf{w},\textbf{b})\in\Bcal_R$. We linearize nonlinear equations (\ref{E3-1}) around $(\textbf{w},\textbf{b})$ to get
the linearized equations with an external force as follows
\bel{E3-2}
\aligned
\textbf{h}_t-\nu\triangle\textbf{h}
&=-\textbf{h}\cdot\nabla(\overline{\textbf{v}}_{\overline{T}^*}+\textbf{w})-(\overline{\textbf{v}}_{\overline{T}^*}+\textbf{w})\cdot\nabla\textbf{h}+\nabla[(\Fcal_{\textbf{w}} p)\textbf{h}+(\Fcal_{\textbf{b}} p)\textbf{q}]+(\overline{\textbf{H}}_{\overline{T}^*}+\textbf{b})\cdot\nabla\textbf{q}]\\
&\quad+\textbf{q}\cdot\nabla(\overline{\textbf{H}}_{\overline{T}^*}+\textbf{b})-\nabla((\overline{\textbf{H}}_{\overline{T}^*}+\textbf{b})\cdot\textbf{q})+\textbf{f}(t,x),
\endaligned
\ee
\bel{E3-2rr1}
\aligned
\textbf{q}_t-\mu\triangle\textbf{q}&=(\overline{\textbf{H}}_{\overline{T}^*}+\textbf{b})\cdot\nabla\textbf{h}+\textbf{q}\cdot\nabla(\overline{\textbf{v}}_{\overline{T}^*}+\textbf{w})
-(\overline{\textbf{v}}_{\overline{T}^*}+\textbf{w})\cdot\nabla\textbf{q}\\
&\quad-\textbf{h}\cdot\nabla(\overline{\textbf{H}}_{\overline{T}^*}+\textbf{b})+\textbf{g}(t,x),~~
\endaligned
\ee
\bel{E3-2rrx}
\nabla\cdot\textbf{h}=0,\quad \nabla\cdot\textbf{q}=0,
\ee
where $(t,x)\in(0,\overline{T}^*)\times\Omega_t$, $\textbf{h}=(h_1,h_2,h_3)^T\in\RR^3$ and $\textbf{q}=(q_1,q_2,q_3)^T\in\RR^3$ denote two unknown vector functions,
$\Fcal_{\textbf{w}}$ and $\Fcal_{\textbf{b}}$ denote the Fr\'{e}chet derivatives on $\textbf{w}$ and $\textbf{b}$, respectively,
$\textbf{f}(t,x)=(f_1(t,x).f_2(t,x),f_3(t,x))^T\in\RR^3$ and $\textbf{g}(t,x)=(g_1(t,x).g_2(t,x),g_3(t,x))^T\in\RR^3$ are two external forces. More precisely, it holds
\bel{E3-2rr2}
\aligned
&(\Fcal_{\textbf{w}} p)\textbf{h}+(\Fcal_{\textbf{b}} p)\textbf{q}=-2\triangle^{-1}\Big[\sum_{i=1}^3\del_{x_i}w_i\del_{x_i}h_i+k(T^*-t)^{2a}(\del_{x_1}h_2-\del_{x_2}h_1)+\sum_{i,j=1,i\neq j}^3\del_{x_i}h_j\del_{x_j}w_i\\
&\quad-\sum_{i=1}^3\del_{x_i}b_i\del_{x_i}q_i-{2\bar{a}k(T^*-t)^{2a+1}\over 4a+1}\Big(x_3(\del_{x_1}q_2-\del_{x_2}q_1)+x_2\del_{x_1}q_3-x_1\del_{x_2}q_3\Big)\\
&\quad-\bar{a}(\del_{x_1}q_1+\del_{x_2}q_2-2\del_{x_3}q_3)-\sum_{i,j=1,i\neq j}^3\del_{x_i}q_j\del_{x_j}b_i\Big]-{1\over2}\sum_{i=1}^3b_iq_i-\bar{a}\Big(x_1q_1+x_2q_2-2x_3q_3\Big)\\
&\quad-{2\bar{a}k(T^*-t)^{2a+1}\over 4a+1}\Big(x_2x_3q_1-x_1x_3q_2\Big).
\endaligned
\ee

We supplement the linearized equations (\ref{E3-2}) with the initial data
\bel{AAAAE3-2R0}
\textbf{h}(0,x)=\textbf{h}_0(x),\quad \textbf{q}(0,x)=\textbf{q}_0(x),
\ee
and the boundary condition
\bel{AAAAE3-2R1}
\textbf{h}(t,x)|_{x\in\del\Omega_t}=0,\quad \textbf{q}(t,x)|_{x\in\del\Omega_t}=0.
\ee

We introduce the similarity coordinates
\bel{Y-1}
\aligned
&\tau=-\ln(\overline{T}^*-t)+\ln \overline{T}^*,\\
&y={x\over \sqrt{\overline{T}^*-t}},\quad \forall x\in\Omega_t,~~\forall y\in\Omega:=[0,1]^3,
\endaligned
\ee
where one can see the blowup time $\overline{T}^*>0$ has been transformed into $+\infty$ in the similarity coordinates (\ref{Y-1}).
Thus the smooth bounded domain $\Omega_t$ is transformed into a fixed domain
$$
\overline{\Omega}:=\{(\tau,y):0<\tau<+\infty,\quad y\in\Omega:=([0,1])^3\}.
$$
So the local existence of linearized coupled system (\ref{E3-2})-(\ref{E3-2rr1}) with the incompressible condition in some Sobolev space is equivalent to prove the global existence of linearized coupled system. More precisely, equations (\ref{E3-2}) is transformed into three coupled equations as follows
\bel{AE3-3-y-1}
\aligned
&\del_{\tau}h_1-\nu\triangle_yh_1-{y\over2}\cdot\nabla_yh_1+ah_1+ay_1\del_{y_1}h_1+ay_2\del_{y_2}h_1-2ay_3\del_{y_3}h_1\\%+2\bar{a}(y_1+T^{{1\over2}}e^{-{\tau\over2}}b_1)\del_{y_1}q_1\\
%+2\bar{a}(y_2+T^{{1\over2}}e^{-{\tau\over2}}b_2)\del_{y_1}q_2
%-4\bar{a}(y_3+T^{{1\over2}}e^{-{\tau\over2}}b_3)\del_{y_1}q_3
&\quad+k(\overline{T}^*)^{2a+1}e^{-(2a+1)\tau}\Big(h_2+y_2\del_{y_1}h_1+y_1\del_{y_2}h_1\Big)
+(\overline{T}^*)^{{1\over2}}e^{-{1\over2}\tau}\sum_{i=1}^3\Big(h_i\del_{y_i}w_1+w_i\del_{y_i}h_1\Big)\\
%+{2\bar{a}k(\overline{T}^*)^{2a+{1\over2}}\over 4a+1}e^{-(2a+{1\over2})\tau}\Big((2y_2y_3+b_1)\del_{y_1}q_1-(2y_1y_3+b_2)\del_{y_1}q_2\Big)\\
&=(\overline{T}^*)^{{1\over2}}\del_{y_1}\overline{f}+\overline{T}^*e^{-\tau}f_1(\overline{T}^*(1-e^{-\tau}),(\overline{T}^*)^{{1\over2}}e^{-{1\over2}\tau}y),
\endaligned
\ee
\bel{AE3-4-y-2}
\aligned
&\del_{\tau}h_2-\nu\triangle_yh_2-{y\over2}\cdot\nabla_yh_2+ah_2+ay_1\del_{y_1}h_2+ay_2\del_{y_2}h_2-2ay_3\del_{y_3}h_2\\
&\quad-k(\overline{T}^*)^{2a+1}e^{-(2a+1)\tau}\Big(h_1-y_2\del_{y_1}h_2+y_1\del_{y_2}h_2\Big)+(\overline{T}^*)^{{1\over2}}e^{-{1\over2}\tau}\sum_{i=1}^3\Big(h_i\del_{y_i}w_2+w_i\del_{y_i}h_2\Big)\\
&=(\overline{T}^*)^{{1\over2}}\del_{y_2}\overline{f}+\overline{T}^*e^{-\tau}f_2(\overline{T}^*(1-e^{-\tau}),(\overline{T}^*)^{{1\over2}}e^{-{1\over2}\tau}y),
\endaligned
\ee
\bel{AE3-5-y-3}
\aligned
&\del_{\tau}h_3-\nu\triangle_yh_3-{y\over2}\cdot\nabla_yh_3-2ah_3+ay_1\del_{y_1}h_3+ay_2\del_{y_2}h_3-2ay_3\del_{y_3}h_3\\
&\quad+k(\overline{T}^*)^{2a+1}e^{-(2a+1)\tau}\Big(y_2\del_{y_1}h_3-y_1\del_{y_2}h_3\Big)+(\overline{T}^*)^{{1\over2}}e^{-{1\over2}\tau}\sum_{i=1}^3\Big(h_i\del_{y_i}w_3+w_i\del_{y_i}h_3\Big)\\
&=(\overline{T}^*)^{{1\over2}}\del_{y_3}\overline{f}+\overline{T}^*e^{-\tau}f_3(\overline{T}^*(1-e^{-\tau}),(\overline{T}^*)^{{1\over2}}e^{-{1\over2}\tau}y),
\endaligned
\ee
and equations (\ref{E3-2rr1}) is transformed into three coupled equations as follows
\bel{AENA3-3}
\aligned
&\del_{\tau}q_1-\mu\triangle_yq_1-{y\over2}\cdot\nabla_yq_1-aq_1+ay_1\del_{y_1}q_1+ay_2\del_{y_2}q_1-2ay_3\del_{y_3}q_1+\bar{a}\overline{T}^*e^{-\tau}h_1\\
&\quad+(\overline{T}^*)^{{1\over2}}e^{-{\tau\over2}}\sum_{i=1}^3(h_i\del_{y_1}b_i-b_i\del_{y_1}h_i)
-\bar{a}\Big(y_1\del_{y_1}h_1+y_2\del_{y_1}h_2-2y_3\del_{y_1}h_3\Big)\\
&\quad+k(\overline{T}^*)^{2a+1}e^{-(2a+1)\tau}\Big(-h_2+y_2\del_{y_1}q_1+y_1\del_{y_2}q_1\Big)
-(\overline{T}^*)^{{1\over2}}e^{-{1\over2}\tau}\sum_{i=1}^3\Big(h_i\del_{y_i}w_1-w_i\del_{y_i}q_1\Big)\\
&\quad+{2\bar{a}k(\overline{T}^*)^{2a+{1\over2}}\over 4a+1}e^{-(2a+{1\over2})\tau}\Big(y_1y_3\del_{y_1}h_2-y_2y_3\del_{y_1}h_1-\overline{T}^*e^{-\tau}y_3h_2\Big)\\
&=\overline{T}^*e^{-\tau}g_1(\overline{T}^*(1-e^{-\tau}),(\overline{T}^*)^{{1\over2}}e^{-{1\over2}\tau}y),
\endaligned
\ee
\bel{AENA3-4}
\aligned
&\del_{\tau}q_2-\mu\triangle_yq_2-{y\over2}\cdot\nabla_yq_2-aq_2+ay_1\del_{y_1}q_2+ay_2\del_{y_2}q_2-2ay_3\del_{y_3}q_2+\bar{a}\overline{T}^*e^{-\tau}h_2\\
&\quad+(\overline{T}^*)^{{1\over2}}e^{-{\tau\over2}}\sum_{i=1}^3(h_i\del_{y_2}b_i-b_i\del_{y_2}h_i)-\bar{a}\Big(y_1\del_{y_2}h_1+y_2\del_{y_2}h_2-2y_3\del_{y_2}h_3\Big)\\
&\quad-k(\overline{T}^*)^{2a+1}e^{-(2a+1)\tau}\Big(-h_1-y_2\del_{y_1}q_2+y_1\del_{y_2}q_2\Big)-(\overline{T}^*)^{{1\over2}}e^{-{1\over2}\tau}\sum_{i=1}^3\Big(h_i\del_{y_i}w_2-w_i\del_{y_i}q_2\Big)\\
&\quad+{2\bar{a}k(\overline{T}^*)^{2a+{1\over2}}\over 4a+1}e^{-(2a+{1\over2})\tau}\Big(y_1y_3\del_{y_2}h_2-y_2y_3\del_{y_2}h_1+\overline{T}^*e^{-\tau}y_3h_1\Big)\\
&=\overline{T}^*e^{-\tau}g_2(\overline{T}^*(1-e^{-\tau}),(\overline{T}^*)^{{1\over2}}e^{-{1\over2}\tau}y),
\endaligned
\ee
\bel{AENAY3-1}
\aligned
&\del_{\tau}q_3-\mu\triangle_yq_3-{y\over2}\cdot\nabla_yq_3+2aq_3+ay_1\del_{y_1}q_3+ay_2\del_{y_2}q_3-2ay_3\del_{y_3}q_3-2\bar{a}\overline{T}^*e^{-\tau}h_3\\
&\quad+(\overline{T}^*)^{{1\over2}}e^{-{\tau\over2}}\sum_{i=1}^3(h_i\del_{y_3}b_i-b_i\del_{y_3}h_i)
-\bar{a}\Big(y_1\del_{y_3}h_1+y_2\del_{y_3}h_2-2y_3\del_{y_3}h_3\Big)\\
&\quad+k(\overline{T}^*)^{2a+1}e^{-(2a+1)\tau}\Big(y_2\del_{y_1}q_3-y_1\del_{y_2}q_3\Big)
-(\overline{T}^*)^{{1\over2}}e^{-{1\over2}\tau}\sum_{i=1}^3\Big(h_i\del_{y_i}w_3-w_i\del_{y_i}q_3\Big)\\
&\quad+{2\bar{a}k(\overline{T}^*)^{2a+{1\over2}}\over 4a+1}e^{-(2a+{1\over2})\tau}\Big(y_1y_3\del_{y_3}h_2-y_2y_3\del_{y_3}h_1+\overline{T}^*e^{-\tau}(y_2h_1-y_1h_2)\Big)\\
&=\overline{T}^*e^{-\tau}g_3(\overline{T}^*(1-e^{-\tau}),(\overline{T}^*)^{{1\over2}}e^{-{1\over2}\tau}y),
\endaligned
\ee
with the incompressible condition
$$
\nabla_y\cdot \textbf{h}=0,\quad \nabla_y\cdot \textbf{q}=0
$$
where
\bel{AE3-5R0-y}
\aligned
\overline{f}&=-2\triangle_y^{-1}\Big[\sum_{i=1}^3\del_{y_i}w_i\del_{y_i}h_i+k(\overline{T}^*)^{2a}e^{-2a\tau}(\del_{y_1}h_2-\del_{y_2}h_1)+\sum_{i,j=1,i\neq j}^3\del_{y_i}h_j\del_{y_j}w_i\\
&\quad-\sum_{i=1}^3\del_{y_i}b_i\del_{y_i}q_i-{2\bar{a}k(\overline{T}^*)^{2a+1}e^{-(2a+1)\tau}\over 4a+1}\Big(y_3(\del_{y_1}q_2-\del_{y_2}q_1)+y_2\del_{y_1}q_3-y_1\del_{y_2}q_3\Big)\\
&\quad-\bar{a}(\overline{T}^*)^{{1\over2}}e^{-{1\over2}\tau}(\del_{y_1}q_1+\del_{y_2}q_2-2\del_{y_3}q_3)-\sum_{i,j=1,i\neq j}^3\del_{y_i}q_j\del_{y_j}b_i\Big]-{1\over2}\overline{T}^*e^{-\tau}\sum_{i=1}^3b_iq_i\\
&\quad-\bar{a}(\overline{T}^*)^{{1\over2}}e^{-{1\over2}\tau}(y_1q_1+y_2q_2-2y_3q_3)-{2\bar{a}k(\overline{T}^*)^{2a+1}e^{-(2a+1)\tau}\over 4a+1}(y_2y_3q_1-y_1y_3q_2).
\endaligned
\ee

We supplement the linearized system (\ref{AE3-3-y-1})-(\ref{AENAY3-1}) with the initial data
\bel{E3-2R0}
\aligned
\left\{
\begin{array}{lll}
&h_i(0,y)=h_{i0}(y),\quad \forall y\in\Omega,\quad i=1,2,3,\\
&q_i(0,y)=q_{i0}(y), \quad \forall y\in\Omega,\quad i=1,2,3,
\end{array}
\right.
\endaligned
\ee
and the boundary condition
\bel{E3-2R1}
\aligned
\left\{
\begin{array}{lll}
&h_i(\tau,y)|_{y\in\del\Omega}=0,\quad i=1,2,3,\\
&q_i(\tau,y)|_{y\in\del\Omega}=0, \quad i=1,2,3.
\end{array}
\right.
\endaligned
\ee

\subsection{Time-decay of solutions in $L^2$-norm for the linear system}

We first derive $L^2$-estimate of solutions to linearized equations (\ref{AE3-3-y-1})-(\ref{AENAY3-1}) with the initial data (\ref{E3-2R0}) and boundary condition (\ref{E3-2R1}).

\begin{lemma}
Let constants $a\in(0,\min\{\nu,\mu\})$ and $\forall s\in\NN^+$.
Assume that $\textbf{f},\textbf{g}\in \CC^1((0,+\infty),H^s(\Omega))$ and $(\textbf{w},\textbf{b})^T\in \Bcal_R$. The solution $(\textbf{h},\textbf{q})^T$ of linearized coupled system  (\ref{AE3-3-y-1})-(\ref{AENAY3-1}) with the initial data (\ref{E3-2R0}) and condition (\ref{E3-2R1}) satisfies
$$
\int_{\Omega}\Big(|\textbf{h}|^2+|\textbf{q}|^2\Big)dy
\lesssim e^{-C_{R,\nu,\mu}\tau}\Big[\int_{\Omega}\Big(|\textbf{h}_0|^2+|\textbf{q}_0|^2\Big)dy+\int_0^{+\infty}\int_{\Omega}\Big(|\textbf{f}|^2+|\textbf{g}|^2\Big)dyd\tau\Big],\quad \forall\tau>0,
$$ 
where $C_{R,\nu,\mu}$ is a positive constant depending on constants $R,\nu,\mu$.
\end{lemma}
\begin{proof}
Multiplying both sides of (\ref{AE3-3-y-1})-(\ref{AE3-5-y-3}) by $h_1,h_2,h_3$, respectively, then integrating by parts, it holds
\bel{E3-6}
\aligned
&{1\over 2}{d\over d\tau}\|h_1\|_{\LL^2(\Omega)}^2+\nu\sum_{i,j=1}^3\|\del_{y_i} h_j\|^2_{\LL^2(\Omega)}+(a+{3\over4})\|h_1\|^2_{\LL^2(\Omega)}\\
&\quad+k(\overline{T}^*)^{2a+1}e^{-(2a+1)\tau}\int_{\Omega}h_1h_2dy
+(\overline{T}^*)^{{1\over2}}e^{-{1\over2}\tau}\sum_{i=1}^3\int_{\Omega}h_1\Big(h_i\del_{y_i}w_1+w_i\del_{y_i}h_1\Big)dy\\
&=(\overline{T}^*)^{{1\over2}}\int_{\Omega}h_1\del_{y_1}\overline{f}dy+\overline{T}^*e^{-\tau}\int_{\Omega}h_1f_1dy,
\endaligned
\ee
\bel{E3-7}
\aligned
&{1\over 2}{d\over d\tau}\|h_2\|_{\LL^2(\Omega)}^2+\nu\sum_{i,j=1}^3\|\del_{y_i} h_j\|^2_{\LL^2(\Omega)}+(a+{3\over4})\|h_2\|^2_{\LL^2(\Omega)}\\
&\quad-k(\overline{T}^*)^{2a+1}e^{-(2a+1)\tau}\int_{\Omega}h_1h_2dy
+(\overline{T}^*)^{{1\over2}}e^{-{1\over2}\tau}\sum_{i=1}^3\int_{\Omega}h_2\Big(h_i\del_{y_i}w_2+w_i\del_{y_i}h_2\Big)dy\\
&=(\overline{T}^*)^{{1\over2}}\int_{\Omega}h_2\del_{y_2}\overline{f}dy+\overline{T}^*e^{-\tau}\int_{\Omega}h_2f_2dy,
\endaligned
\ee
and
\bel{E3-8}
\aligned
&{1\over 2}{d\over d\tau}\|h_3\|_{\LL^2(\Omega)}^2+\nu\sum_{i,j=1}^3\|\del_{y_i} h_j\|^2_{\LL^2(\Omega)}+({3\over4}-2a)\|h_3\|^2_{\LL^2(\Omega)}\\
&\quad+(\overline{T}^*)^{{1\over2}}e^{-{1\over2}\tau}\sum_{i=1}^3\int_{\Omega}h_3\Big(h_i\del_{y_i}w_3+w_i\del_{y_i}h_3\Big)dy\\
&=(\overline{T}^*)^{{1\over2}}\int_{\Omega}h_3\del_{y_3}\overline{f}dy+\overline{T}^*e^{-\tau}\int_{\Omega}h_3f_3dy.
\endaligned
\ee

Similarly, we multiply both sides of (\ref{AENA3-3})-(\ref{AENAY3-1}) with $h_1,h_2,h_3$ and $q_1,q_2,q_3$, then we integrate by parts to derive
\bel{AAE3-6}
\aligned
&{1\over 2}{d\over d\tau}\|q_1\|_{\LL^2(\Omega)}^2+\mu\sum_{i,j=1}^3\|\del_{y_i} q_j\|^2_{\LL^2(\Omega)}+({3\over4}-a)\|q_1\|^2_{\LL^2(\Omega)}+\bar{a}\overline{T}^*e^{-\tau}\int_{\Omega}q_1h_1dy
\\
&\quad+(\overline{T}^*)^{{1\over2}}e^{-{\tau\over2}}\sum_{i=1}^3\int_{\Omega}(h_i\del_{y_1}b_i-b_i\del_{y_1}h_i)q_1dy
-\bar{a}\int_{\Omega}\Big(y_1\del_{y_1}h_1+y_2\del_{y_1}h_2-2y_3\del_{y_1}h_3\Big)q_1dy\\\
&\quad-k(\overline{T}^*)^{2a+1}e^{-(2a+1)\tau}\int_{\Omega}h_2q_1dy
-(\overline{T}^*)^{{1\over2}}e^{-{1\over2}\tau}\sum_{i=1}^3\int_{\Omega}\Big(h_i\del_{y_i}w_1-w_i\del_{y_i}q_1\Big)q_1dy\\
&\quad+{2\bar{a}k(\overline{T}^*)^{2a+{1\over2}}\over 4a+1}e^{-(2a+{1\over2})\tau}\int_{\Omega}\Big(y_1y_3\del_{y_1}h_2-y_2y_3\del_{y_1}h_1-\overline{T}^*e^{-\tau}y_3h_2\Big)q_1dy\\
&=\overline{T}^*e^{-\tau}\int_{\Omega}g_1q_1dy,
\endaligned
\ee
\bel{AAE3-7}
\aligned
&{1\over 2}{d\over d\tau}\|q_2\|_{\LL^2(\Omega)}^2+\mu\sum_{i,j=1}^3\|\del_{y_i} q_j\|^2_{\LL^2(\Omega)}+({3\over4}-a)\|q_2\|^2_{\LL^2(\Omega)}
+\bar{a}\overline{T}^*e^{-\tau}\int_{\Omega}q_2h_2dy
\\
&\quad+(\overline{T}^*)^{{1\over2}}e^{-{\tau\over2}}\sum_{i=1}^3\int_{\Omega}(h_i\del_{y_1}b_i-b_i\del_{y_1}h_i)q_2dy
-\bar{a}\int_{\Omega}\Big(y_1\del_{y_2}h_1+y_2\del_{y_2}h_2-2y_3\del_{y_2}h_3\Big)q_2dy\\\
&\quad+k(\overline{T}^*)^{2a+1}e^{-(2a+1)\tau}\int_{\Omega}h_1q_2dy
-(\overline{T}^*)^{{1\over2}}e^{-{1\over2}\tau}\sum_{i=1}^3\int_{\Omega}\Big(h_i\del_{y_i}w_2-w_i\del_{y_i}q_2\Big)q_2dy\\
&\quad+{2\bar{a}k(\overline{T}^*)^{2a+{1\over2}}\over 4a+1}e^{-(2a+{1\over2})\tau}\int_{\Omega}\Big(y_1y_3\del_{y_2}h_2-y_2y_3\del_{y_2}h_1+\overline{T}^*e^{-\tau}y_3h_1\Big)q_2dy\\
&=\overline{T}^*e^{-\tau}\int_{\Omega}g_2q_2dy,
\endaligned
\ee
and
\bel{AAE3-8}
\aligned
&{1\over 2}{d\over d\tau}\|q_3\|_{\LL^2(\Omega)}^2+\mu\sum_{i,j=1}^3\|\del_{y_i} q_j\|^2_{\LL^2(\Omega)}+({3\over4}+2a)\|q_3\|^2_{\LL^2(\Omega)}
-2\bar{a}\overline{T}^*e^{-\tau}\int_{\Omega}q_3h_3dy\\
&\quad+(\overline{T}^*)^{{1\over2}}e^{-{\tau\over2}}\sum_{i=1}^3\int_{\Omega}(h_i\del_{y_1}b_i-b_i\del_{y_1}h_i)q_3dy
-\bar{a}\int_{\Omega}\Big(y_1\del_{y_3}h_1+y_2\del_{y_3}h_2-2y_3\del_{y_3}h_3\Big)q_3dy\\
&\quad
-(\overline{T}^*)^{{1\over2}}e^{-{1\over2}\tau}\sum_{i=1}^3\int_{\Omega}\Big(h_i\del_{y_i}w_3-w_i\del_{y_i}q_3\Big)q_3dy\\
&\quad+{2\bar{a}k(\overline{T}^*)^{2a+{1\over2}}\over 4a+1}e^{-(2a+{1\over2})\tau}\int_{\Omega}\Big(y_1y_3\del_{y_3}h_2-y_2y_3\del_{y_3}h_1+\overline{T}^*e^{-\tau}(y_2h_1-y_1h_2)\Big)q_3dy\\
&=\overline{T}^*e^{-\tau}\int_{\Omega}g_3q_3dy,
\endaligned
\ee

We sum up (\ref{E3-6})-(\ref{AAE3-8}) to get
\bel{E3-6-y-0}
\aligned
&{1\over 2}\sum_{i=1}^3{d\over d\tau}\Big(\|h_i\|_{\LL^2(\Omega)}^2+\|q_i\|^2_{\LL^2(\Omega)}\Big)+3\sum_{i,j=1}^3\Big(\nu\|\del_{y_i} h_j\|^2_{\LL^2(\Omega)}+\mu|\del_{y_i} q_j\|^2_{\LL^2(\Omega)}\Big)\\
&\quad+(a+{3\over4})\Big(\|h_1\|^2_{\LL^2(\Omega)}+\|h_2\|^2_{\LL^2(\Omega)}\Big)+({3\over4}-2a)\|h_3\|^2_{\LL^2(\Omega)}
+({3\over4}-a)\Big(\|q_1\|^2_{\LL^2(\Omega)}+\|q_2\|^2_{\LL^2(\Omega)}\Big)\\
&\quad+({3\over4}+2a)\|q_3\|^2_{\LL^2(\Omega)}+(\overline{T}^*)^{{1\over2}}e^{-{\tau\over2}}\sum_{j=1}^3\sum_{i=1}^3\int_{\Omega}\Big(h_j(h_i\del_{y_i}w_j+w_i\del_{y_i}h_j)+(h_i\del_{y_j}b_i-b_i\del_{y_j}h_i)q_j\Big)dy\\
&\quad-\bar{a}\sum_{i=1}^3\int_{\Omega}\Big(y_1\del_{y_i}h_1+y_2\del_{y_i}h_2-2y_3\del_{y_i}h_3\Big)q_idy-k(\overline{T}^*)^{2a+1}e^{-(2a+1)\tau}\int_{\Omega}\Big(h_2q_1+h_1q_2\Big)dy\\
&\quad-(\overline{T}^*)^{{1\over2}}e^{-{1\over2}\tau}\sum_{j=1}^3\sum_{i=1}^3\int_{\Omega}\Big(h_i\del_{y_i}w_j-w_i\del_{y_i}q_j\Big)q_jdy\\
&\quad-{2\bar{a}k(\overline{T}^*)^{2a+{3\over2}}\over 4a+1}e^{-(2a+{3\over2})\tau}\int_{\Omega}\Big(y_3(-h_2q_1+h_1q_2)+(y_2h_1-y_1h_2)q_3\Big)dy
\\
&\quad+{2\bar{a}k(\overline{T}^*)^{2a+{1\over2}}\over 4a+1}e^{-(2a+{1\over2})\tau}\sum_{i=1}^3\int_{\Omega}\Big(y_1y_3\del_{y_i}h_2-y_2y_3\del_{y_i}h_1\Big)q_idy\\
&=(\overline{T}^*)^{{1\over2}}\sum_{i=1}^3\int_{\Omega}h_i\del_{y_i}\overline{f}dy+\overline{T}^*e^{-\tau}\sum_{i=1}^3\int_{\Omega}\Big(h_if_i+q_ig_i\Big)dy.
\endaligned
\ee

We now estimate each coupled nonlinear term in (\ref{E3-6-y-0}). Note that $y\in\Omega$, $(\textbf{w},\textbf{b})\in \Bcal_R$ and $H^{{3}}(\Omega)\subset L^{\infty}(\Omega)$. We use Young's inequality to derive
\bel{E3-9}
\aligned
&\Big|\sum_{j=1}^3\sum_{i=1}^3\int_{\Omega}\Big(h_j(h_i\del_{y_i}w_j+w_i\del_{y_i}h_j)+(h_i\del_{y_j}b_i-b_i\del_{y_j}h_i)q_j\Big)\Big|\\
&\lesssim \Big(\sum_{j=1}^3\sum_{i=1}^3\|\del_{y_i}w_j\|^2_{\LL^{\infty}(\Omega)}+\|w_i\|^2_{\LL^{\infty}(\Omega)}
+\|\del_{y_j}q_i\|^2_{\LL^{\infty}(\Omega)}+\|q_i\|^2_{\LL^{\infty}(\Omega)}\Big)\\
&\quad\times\Big(\sum_{i=1}^3(\|h_i\|^2_{\LL^2(\Omega)}+\|q_i\|^2_{\LL^2(\Omega)})+\sum_{j=1}^3\sum_{i=1}^3\|\del_{y_i}h_j\|^2_{\LL^2(\Omega)}\Big),\\
&\lesssim C_R\Big(\sum_{i=1}^3(\|h_i\|^2_{\LL^2(\Omega)}+\|q_i\|^2_{\LL^2(\Omega)})+\sum_{j=1}^3\sum_{i=1}^3\|\del_{y_i}h_j\|^2_{\LL^2(\Omega)}\Big),
\endaligned
\ee
Similarly, it holds
\bel{AE3-9}
\aligned
&\Big|\sum_{i=1}^3\int_{\Omega}\Big(y_1\del_{y_i}h_1+y_2\del_{y_i}h_2-2y_3\del_{y_i}h_3\Big)q_idy\Big|\lesssim  C_R\sum_{i=1}^3\Big(\|q_i\|^2_{\LL^2(\Omega)}+\sum_{j=1}^3\|\del_{y_i}h_j\|^2_{\LL^2(\Omega)}\Big),\\
&\Big|\int_{\Omega}\Big(h_2q_1+h_1q_2\Big)dy\Big|\leq {1\over2}\sum_{i=1}^2\Big(\|h_i\|^2_{\LL^2(\Omega)}+\|q_i\|^2_{\LL^2(\Omega)}\Big),
\endaligned
\ee
and
\bel{aAE3-9}
\aligned
&\Big|\sum_{j=1}^3\sum_{i=1}^3\int_{\Omega}\Big(h_i\del_{y_i}w_j-w_i\del_{y_i}q_j\Big)q_jdy\Big|\lesssim C_R\sum_{i=1}^2\Big(\|h_i\|^2_{\LL^2(\Omega)}+\|q_i\|^2_{\LL^2(\Omega)}\Big),\\
&\Big|\int_{\Omega}\Big(y_3(-h_2q_1+h_1q_2)+(y_2h_1-y_1h_2)q_3\Big)dy\Big|\leq{1\over2}\Big(\sum_{i=1}^3\|q_i\|^2_{\LL^2(\Omega)}+\sum_{j=1}^2\|h_j\|^2_{\LL^2(\Omega)}\Big),\\
&\Big|\sum_{i=1}^3\int_{\Omega}\Big(y_1y_3\del_{y_i}h_2-y_2y_3\del_{y_i}h_1\Big)q_idy\Big|\leq{1\over2}\sum_{i=1}^3\Big(\|q_i\|^2_{\LL^2(\Omega)}+\sum_{j=1}^2\|\del_{y_i}h_j\|^2_{\LL^2(\Omega)}\Big),
\endaligned
\ee
and
\bel{E3-10-y-0}
\Big|\sum_{i=1}^3\int_{\RR^3}\Big(h_if_i+q_ig_i\Big)dy\Big|\leq{1\over2}\sum_{i=1}^3\Big(\|h_i\|^2_{\LL^2(\Omega)}+\|q_i\|^2_{\LL^2(\Omega)}+\|f_i\|^2_{\LL^2(\Omega)}+\|g_i\|^2_{\LL^2(\Omega)}\Big).
\ee
where $C_R$ is a postive constant depending on $R$.

On the other hand, by (\ref{AE3-5R0-y}) and incompressible condition, we integrate by parts to derive
%the standard Calderon-Zygmund theory, i.e. for Riesz operator $\Rcal$, there is $\|\Rcal w\|_{\LL^p}\leq\|w\|_{\LL^p}$ with $1<p<\infty$,
\bel{E3-10}
\sum_{i=1}^3\int_{\Omega}h_i\del_{y_i}\overline{f}dy=\int_{\Omega}(\sum_{i=1}^3\del_{y_i}h_i)\overline{f}dy=0.
\ee

Thus by (\ref{E3-9})-(\ref{E3-10-y-0}), it follows from (\ref{E3-6-y-0}) that
\bel{E3-11}
\aligned
{1\over 2}\sum_{i=1}^3{d\over d\tau}\Big(\|h_i\|_{\LL^2(\Omega)}^2+\|q_i\|^2_{\LL^2(\Omega)}\Big)&+\sum_{i,j=1}^3\Big((3\nu-{1\over2}-C_{R})\|\del_{y_i} h_j\|^2_{\LL^2(\Omega)}+(3\mu-C_R)|\del_{y_i} q_j\|^2_{\LL^2(\Omega)}\Big)\\
&+(a-{3\over4}-C_R)\Big(\|h_1\|^2_{\LL^2(\Omega)}+\|h_2\|^2_{\LL^2(\Omega)}\Big)-(2a+C_R)\|h_3\|^2_{\LL^2(\Omega)}\\
&-({5\over4}+a+C_R)\Big(\|q_1\|^2_{\LL^2(\Omega)}+\|q_2\|^2_{\LL^2(\Omega)}\Big)
+(-{5\over4}+2a-C_R)\|q_3\|^2_{\LL^2(\Omega)}\\
&\lesssim{1\over2}\sum_{i=1}^3\Big(\|f_i\|^2_{\LL^2(\Omega)}+\|g_i\|^2_{\LL^2(\Omega)}\Big).
\endaligned
\ee
where $C_{R,\overline{T}^*}$ is a postive constant depending on $R$, which is small as $R$ small.

Since we use Poincar\'{e} inequality to derive
$$
\aligned
\sum_{i,j=1}^3\|\del_{y_i} h_j\|^2_{\LL^2(\Omega)}\gtrsim \sum_{i=1}^3\|h_i\|^2_{\LL^2(\Omega)},\\
\sum_{i,j=1}^3\|\del_{y_i} q_j\|^2_{\LL^2(\Omega)}\gtrsim \sum_{i=1}^3\|q_i\|^2_{\LL^2(\Omega)},
\endaligned
$$
it holds
\bel{AEA3-11}
\aligned
{1\over 2}\sum_{i=1}^3{d\over d\tau}\Big(\|h_i\|_{\LL^2(\Omega)}^2+\|q_i\|^2_{\LL^2(\Omega)}\Big)&+(3\nu+a-{7\over4}-C_R)\Big(\|h_1\|^2_{\LL^2(\Omega)}+\|h_2\|^2_{\LL^2(\Omega)}\Big)\\
&\quad+(3\nu-2a-{1\over2}-C_R)\|h_3\|^2_{\LL^2(\Omega)}\\
&\quad+(3\mu-{5\over4}-a-C_R)\Big(\|q_1\|^2_{\LL^2(\Omega)}+\|q_2\|^2_{\LL^2(\Omega)}\Big)\\
&\quad+(3\mu-{5\over4}+2a-C_R)\|q_3\|^2_{\LL^2(\Omega)}\\
&\lesssim{1\over2}\sum_{i=1}^3\Big(\|f_i\|^2_{\LL^2(\Omega)}+\|g_i\|^2_{\LL^2(\Omega)}\Big).
\endaligned
\ee

Let $0<a<\min\{\nu,\mu\}$.
We notice that we can choose a sufficient small positive constant $R\ll\min\{1,\nu,\mu\}$ and sufficient big $\nu$ and $\mu$ such that
$$
\aligned
&3\nu+a-{7\over4}-C_R>0,\\
&3\nu-2a-{1\over2}-C_R>0,\\
&3\mu-{5\over4}-a-C_R>0,\\
&3\mu-{5\over4}+2a-C_R>0.
\endaligned
$$

Hence, applying Gronwall's inequality to (\ref{AEA3-11}), there exists a positive constant $C_{R,\nu,\mu}$ depending on $R,\nu,\mu$ such that
$$
\sum_{i=1}^3\Big(\|h_i\|_{\LL^2(\Omega)}^2+\|q_i\|_{\LL^2(\Omega)}^2\Big)\lesssim e^{-C_{R,\nu,\mu}\tau}\sum_{i=1}^3\int_{\Omega}\Big[h_{0i}^2
+q_{0i}^2+\int_0^{+\infty}
\Big(f_i^2+g_i^2\Big)d\tau\Big]dy,\quad \forall \tau>0.
$$

\end{proof}

\subsection{Time-decay of solutions in $H^s$-norm for the linear system}

In what follows, we plan to carry out a higher order derivative estimates to the solutions of linearized system (\ref{AE3-3-y-1})-(\ref{AENAY3-1}).
For a fixed integer $s\geq1$, 
applying $\nabla_y^{s}:=\Big(\del_{y_1}^s,\del_{y_2}^s,\del_{y_3}^s\Big)^T$ to both sides of (\ref{AE3-3-y-1})-(\ref{AENAY3-1}), we obtain
\bel{AAE3-3-y-1}
\aligned
&\del_{\tau}\nabla_y^sh_1-\nu\triangle_y\nabla_y^sh_1-{y\over2}\cdot\del_y\nabla_y^sh_1+\Big(a(1+s)-{s\over2}\Big)\nabla_y^sh_1+ay_1\del_{y_1}\nabla_y^sh_1+ay_2\del_{y_2}\nabla_y^sh_1\\
&\quad-2ay_3\del_{y_3}\nabla_y^sh_1-3as\Big(0,0,\del_{y_3}^sh_1\Big)^T
%+(\overline{T}^*)^{{1\over2}}e^{-{1\over2}\tau}\sum_{i=1}^3\Big(\nabla_y^sh_i\del_{y_i}w_1+w_i\del_{y_i}\nabla_y^sh_1\Big)\\
+k(\overline{T}^*)^{2a+1}e^{-(2a+1)\tau}\Big(\nabla_y^sh_2+y_2\del_{y_1}\nabla_y^sh_1+y_1\del_{y_2}\nabla_y^sh_1\Big)\\
&\quad+ks(\overline{T}^*)^{2a+1}e^{-(2a+1)\tau}\Big(\del_{y_2}\del_{y_1}^{s-1}h_1,\del_{y_1}\del_{y_2}^{s-1}h_1,0\Big)^T=\tilde{f}_1,
\endaligned
\ee
\bel{AAE3-4-y-2}
\aligned
&\del_{\tau}\nabla^s_yh_2-\nu\triangle_y\nabla^s_yh_2-{y\over2}\cdot\del_y\nabla^s_yh_2+\Big(a(s+1)-{s\over2}\Big)\nabla^s_yh_2+ay_1\del_{y_1}\nabla^s_yh_2+ay_2\del_{y_2}\nabla^s_yh_2\\
&\quad-2ay_3\del_{y_3}\nabla^s_yh_2-3as\Big(0,0,\del_{y_3}^sh_2\Big)^T
%+(\overline{T}^*)^{{1\over2}}e^{-{1\over2}\tau}\sum_{i=1}^3\Big(\nabla_y^sh_i\del_{y_i}w_2+w_i\del_{y_i}\nabla_y^sh_2\Big)\\
+k(\overline{T}^*)^{2a+1}e^{-(2a+1)\tau}\Big(-\nabla^s_yh_1
+y_2\del_{y_1}\nabla^s_yh_2-y_1\del_{y_2}\nabla^s_yh_2\Big)\\
&\quad+ks(\overline{T}^*)^{2a+1}e^{-(2a+1)\tau}\Big(-\del_{y_2}\del_{y_1}^{s-1}h_2,\del_{y_1}\del_{y_2}^{s-1}h_2,0\Big)^T
=\tilde{f}_2,
\endaligned
\ee
\bel{AAE3-5-y-3}
\aligned
&\del_{\tau}\nabla^s_yh_3-\nu\triangle_y\nabla^s_yh_3-{y\over2}\cdot\del_y\nabla^s_yh_3+\Big(a(s-2)-{s\over2}\Big)\nabla^s_yh_3+ay_1\del_{y_1}\nabla^s_yh_3+ay_2\del_{y_2}\nabla^s_yh_3\\
&\quad-2ay_3\del_{y_3}\nabla^s_yh_3-3as\Big(0,0,\del_{y_3}^sh_3\Big)^T
%+(\overline{T}^*)^{{1\over2}}e^{-{1\over2}\tau}\sum_{i=1}^3\Big(\nabla_y^sh_i\del_{y_i}w_3+w_i\del_{y_i}\nabla_y^sh_3\Big)\\
+k(\overline{T}^*)^{2a+1}e^{-(2a+1)\tau}\Big(y_2\del_{y_1}\nabla^s_yh_3-y_1\del_{y_2}\nabla^s_yh_3\Big)\\
&\quad+ks(\overline{T}^*)^{2a+1}e^{-(2a+1)\tau}\Big(-\del_{y_2}\del_{y_1}^{s-1}h_3,\del_{y_1}\del_{y_2}^{s-1}h_3,0\Big)^T=\tilde{f}_3,
\endaligned
\ee
and 
\bel{AAENA3-3}
\aligned
&\del_{\tau}\nabla_y^{s}q_1-\mu\triangle_y\nabla_y^{s}q_1-{y\over2}\cdot\del_y\nabla_y^{s}q_1+\Big(a(s-1)-{s\over2}\Big)\nabla_y^{s}q_1+ay_1\del_{y_1}\nabla_y^{s}q_1+ay_2\del_{y_2}\nabla_y^{s}q_1-2ay_3\del_{y_3}\nabla_y^{s}q_1\\
&\quad-3as\Big(0,0\del_{y_3}^sq_1\Big)^T+\bar{a}\overline{T}^*e^{-\tau}\nabla_y^{s}h_1-\bar{a}\Big(y_1\del_{y_1}\nabla_y^{s}h_1+y_2\del_{y_1}\nabla_y^{s}h_2-2y_3\del_{y_1}\nabla_y^{s}h_3\Big)\\
&\quad-\bar{a}s\Big(\del_{y_1}^sh_1,\del_{y_1}^sh_2,-2\del_{y_1}^sh_3\Big)^T
+k(\overline{T}^*)^{2a+1}e^{-(2a+1)\tau}\Big[-\nabla_y^{s}h_2+y_2\del_{y_1}\nabla_y^{s}q_1+y_1\del_{y_2}\nabla_y^{s}q_1\Big]\\
&\quad+ks(\overline{T}^*)^{2a+1}e^{-(2a+1)\tau}\Big(\del_{y_2}\del_{y_1}^{s-1}q_1,\del_{y_1}\del_{y_2}^{s-1}q_1,0\Big)^T\\
&\quad+{2\bar{a}k(\overline{T}^*)^{2a+{1\over2}}\over 4a+1}e^{-(2a+{1\over2})\tau}\Big(y_1y_3\del_{y_1}\nabla_y^sh_2-y_2y_3\del_{y_1}\nabla_y^sh_1-\overline{T}^*e^{-\tau}y_3\nabla_y^sh_2\Big)\\
&\quad+{2\bar{a}k(\overline{T}^*)^{2a+{1\over2}}\over 4a+1}se^{-(2a+{1\over2})\tau}\Big(y_3\del_{y_1}^sh_2,-y_3\del_{y_1}\del_{y_2}^{s-1}h_1,y_1\del_{y_1}\del_{y_3}^{s-1}h_2-y_2\del_{y_1}\del_{y_3}^{s-1}h_1-\overline{T}^*e^{-\tau}\del_{y_3}^{s-1}h_2\Big)^T=\tilde{g}_1,
\endaligned
\ee
\bel{AAENA3-4}
\aligned
&\del_{\tau}\nabla_y^{s}q_2-\mu\triangle_y\nabla_y^{s}q_2-{y\over2}\cdot\del_y\nabla_y^{s}q_2+\Big(a(s-1)-{s\over2}\Big)\nabla_y^{s}q_2+ay_1\del_{y_1}\nabla_y^{s}q_2+ay_2\del_{y_2}\nabla_y^{s}q_2-2ay_3\del_{y_3}\nabla_y^{s}q_2\\
&\quad-3as\Big(0,0,\del_{y_3}^sq_2\Big)^T+\bar{a}\overline{T}^*e^{-\tau}\nabla_y^{s}h_2-\bar{a}\Big(y_1\del_{y_2}\nabla_y^sh_1+y_2\del_{y_2}\nabla_y^sh_2-2y_3\del_{y_2}\nabla_y^sh_3\Big)\\
&\quad-\bar{a}s\Big(\del_{y_2}^sh_1,\del_{y_2}^sh_2,-2\del_{y_2}^sh_3\Big)^T+k(\overline{T}^*)^{2a+1}e^{-(2a+1)\tau}\Big(\nabla_y^{s}h_1+y_2\del_{y_1}\nabla_y^{s}q_2-y_1\del_{y_2}\nabla_y^{s}q_2\Big)\\
&\quad+k(\overline{T}^*)^{2a+1}e^{-(2a+1)\tau}s\Big(\del_{y_2}\del_{y_1}^{s-1}q_2,-\del_{y_1}\del_{y_2}^{s-1}q_2,0\Big)^T\\
&\quad+{2\bar{a}k(\overline{T}^*)^{2a+{1\over2}}\over 4a+1}e^{-(2a+{1\over2})\tau}\Big(y_1y_3\del_{y_2}\nabla_y^sh_2-y_2y_3\del_{y_2}\nabla_y^sh_1+\overline{T}^*e^{-\tau}y_3\nabla_y^sh_1\Big)\\
&\quad+{2\bar{a}k(\overline{T}^*)^{2a+{1\over2}}\over 4a+1}se^{-(2a+{1\over2})\tau}\Big(y_3\del_{y_2}\del_{y_1}^{s-1}h_2,-y_3\del_{y_2}^{s}h_1,y_1\del_{y_2}\del_{y_3}^{s-1}h_2-y_2\del_{y_2}\del_{y_3}^{s-1}h_1+\overline{T}^*e^{-\tau}\del_{y_3}^{s-1}h_1\Big)^T=\tilde{g}_2,
\endaligned
\ee
\bel{AAENAY3-1}
\aligned
&\del_{\tau}\nabla_y^{s}q_3-\mu\triangle_y\nabla_y^{s}q_3-{y\over2}\cdot\del_y\nabla_y^{s}q_3+\Big(a(s+2)-{s\over2}\Big)\nabla_y^{s}q_3+ay_1\del_{y_1}\nabla_y^{s}q_3+ay_2\del_{y_2}\nabla_y^{s}q_3-2ay_3\del_{y_3}\nabla_y^{s}q_3\\
&\quad-3as\Big(0,0,\del_{y_3}^sq_3\Big)^T-2\bar{a}\overline{T}^*e^{-\tau}\nabla_y^sh_3
-\bar{a}\Big(y_1\del_{y_3}\nabla_y^sh_1+y_2\del_{y_3}\nabla_y^sh_2-2y_3\del_{y_3}\nabla_y^sh_3\Big)\\
&\quad-\bar{a}s\Big(\del_{y_3}^sh_1,\del_{y_3}^sh_2,-2\del_{y_3}^sh_3\Big)^T+k(\overline{T}^*)^{2a+1}e^{-(2a+1)\tau}\Big(y_2\del_{y_1}\nabla_y^{s}q_3-y_1\del_{y_2}\nabla_y^{s}q_3\Big)\\
&\quad+k(\overline{T}^*)^{2a+1}e^{-(2a+1)\tau}s\Big(-\del_{y_2}\del_{y_1}^{s-1}q_3,\del_{y_1}\del_{y_2}^{s-1}q_3,0\Big)^T\\
&\quad+{2\bar{a}k(\overline{T}^*)^{2a+{1\over2}}\over 4a+1}e^{-(2a+{1\over2})\tau}\Big(y_1y_3\del_{y_3}\nabla_y^sh_2-y_2y_3\del_{y_3}\nabla_y^sh_1+\overline{T}^*e^{-\tau}(y_2\nabla_y^sh_1-y_1\nabla_y^sh_2)\Big)\\
&\quad+{2\bar{a}k(\overline{T}^*)^{2a+{1\over2}}\over 4a+1}se^{-(2a+{1\over2})\tau}\Big(y_3\del_{y_3}\del_{y_1}^{s-1}h_2,-y_3\del_{y_3}\del_{y_2}^{s-1}h_1,\del_{y_3}^{s}h_2-y_2\del_{y_3}^{s}h_1+\overline{T}^*e^{-\tau}\del_{y_3}^{s-1}h_1\Big)^T\\
&\quad+{2\bar{a}k(\overline{T}^*)^{2a+{3\over2}}\over 4a+1}e^{-(2a+{3\over2})\tau}\Big(-\del_{y_1}^{s-1}h_2,\del_{y_2}^{s-1}h_1\Big)^T=\tilde{g}_3,
\endaligned
\ee
where we denote by $\nabla^{i_2}_yh\nabla^{i_1}_yw:=\Big(\del_{y_1}^{i_2}h\del_{y_1}^{i_2}w,\del_{y_2}^{i_2}h\del_{y_2}^{i_2}w,\del_{y_3}^{i_2}h\del_{y_3}^{i_2}w\Big)^T$ for convenience, and
\bel{AE3-5RR0}
\tilde{f}_1:=(\overline{T}^*)^{{1\over2}}\del_{y_1}\nabla_y^s\overline{f}+\overline{T}^*e^{-\tau}\nabla_y^sf_1-(\overline{T}^*)^{{1\over2}}e^{-{1\over2}\tau}\sum_{i_1+i_2=s,~0\leq i_2\leq s}\sum_{j=1}^3
\Big(\nabla^{i_2}_yh_j\del_{y_j}\nabla^{i_1}_yw_1+\nabla^{i_1}_yw_j\del_{y_j}\nabla_y^{i_2}h_1\Big),
%&\quad\quad+\nabla_y^{i_2}h_2\del_{y_2}\nabla_y^{i_1}w_1+\nabla_y^{i_1}w_2\del_{y_2}\nabla_y^{i_2}h_1+\nabla_y^{i_2}h_3\del_{y_3}\nabla_y^{i_1}w_1
%+\nabla_y^{i_1}w_3\del_{y_3}\nabla_y^{i_2}h_1\Big),
\ee
\bel{AE3-5RR1}
\tilde{f}_2:=(\overline{T}^*)^{{1\over2}}\del_{y_2}\nabla^s_y\overline{f}+\overline{T}^*e^{-\tau}\nabla^s_yf_2
-(\overline{T}^*)^{{1\over2}}e^{-{1\over2}\tau}\sum_{i_1+i_2=s,~0\leq i_2\leq s}\sum_{j=1}^3
\Big(\nabla^{i_2}_yh_j\del_{y_j}\nabla^{i_1}_yw_2+\nabla^{i_1}_yw_j\del_{y_j}\nabla^{i_2}_yh_2\Big),
%&\quad+\nabla^{i_2}_yh_2\del_{y_2}\nabla^{i_1}_yw_2
%+\nabla^{i_1}_yw_2\del_{y_2}\nabla^{i_2}_yh_2+\nabla^{i_2}_yh_3\del_{y_3}\nabla^{i_1}_yw_2+\nabla^{i_1}_yw_3\del_{y_3}\nabla^{i_2}_yh_2\Big),
\ee
\bel{AE3-5RR2}
\tilde{f}_3:=(\overline{T}^*)^{{1\over2}}\del_{y_3}\nabla^s_y\overline{f}+\overline{T}^*e^{-\tau}\nabla^s_yf_3
-(\overline{T}^*)^{{1\over2}}e^{-{1\over2}\tau}\sum_{i_1+i_2=s,~0\leq i_2\leq s}\sum_{j=1}^3\Big(\nabla^{i_2}_yh_j\del_{y_j}\nabla^{i_1}_yw_3+\nabla^{i_1}_yw_j\del_{y_j}\nabla^{i_2}_yh_3\Big),
%&\quad+\nabla^{i_2}_yh_2\del_{y_2}\nabla^{i_1}_yw_3
%+\nabla^{i_1}_yw_2\del_{y_2}\nabla^{i_2}_yh_3+\nabla^{i_2}_yh_3\del_{y_3}\nabla^{i_1}_yw_3+\nabla^{i_1}_yw_3\del_{y_3}\nabla^{i_2}_yh_3\Big).
\ee
\bel{AAAE3-5RR0}
\aligned
\tilde{g}_1&:=\overline{T}^*e^{-\tau}\nabla_y^sg_1+(\overline{T}^*)^{{1\over2}}e^{-{1\over2}\tau}\sum_{i_1+i_2=s,~0\leq i_2\leq s}\sum_{j=1}^3\Big(\nabla_y^{i_1}h_j\del_{y_j}\nabla_y^{i_2}w_1-\nabla_y^{i_1}w_j\del_{y_j}\nabla_y^{i_2}q_1\Big)\\
&\quad-(\overline{T}^*)^{{1\over2}}e^{-{1\over2}\tau}\sum_{i_1+i_2=s,~0\leq i_2\leq s}\sum_{j=1}^3
\Big(\nabla^{i_1}_yh_j\del_{y_1}\nabla^{i_2}_yb_j-\nabla^{i_1}_yb_j\del_{y_1}\nabla_y^{i_2}h_j\Big),
%&\quad\quad+\nabla_y^{i_2}h_2\del_{y_2}\nabla_y^{i_1}w_1+\nabla_y^{i_1}w_2\del_{y_2}\nabla_y^{i_2}h_1+\nabla_y^{i_2}h_3\del_{y_3}\nabla_y^{i_1}w_1
%+\nabla_y^{i_1}w_3\del_{y_3}\nabla_y^{i_2}h_1\Big),
\endaligned
\ee
\bel{AAAE3-5RR1}
\aligned
\tilde{g}_2&:=\overline{T}^*e^{-\tau}\nabla_y^sg_2+(\overline{T}^*)^{{1\over2}}e^{-{1\over2}\tau}\sum_{i_1+i_2=s,~0\leq i_2\leq s}\sum_{j=1}^3\Big(\nabla_y^{i_1}h_j\del_{y_j}\nabla_y^{i_2}w_2-\nabla_y^{i_1}w_j\del_{y_j}\nabla_y^{i_2}q_2\Big)\\
&\quad-(\overline{T}^*)^{{1\over2}}e^{-{\tau\over2}}\sum_{i_1+i_2=s,~0\leq i_2\leq s}\sum_{j=1}^3\Big(\nabla_y^{i_1}h_j\del_{y_2}\nabla_y^{i_2}b_j-\nabla_y^{i_1}b_j\del_{y_2}\nabla_y^{i_2}h_j\Big),
%&\quad-(\overline{T}^*)^{{1\over2}}e^{-{1\over2}\tau}\sum_{i_1+i_2=s,~0\leq i_2\leq s-1}\sum_{j=1}^3
%\Big(\nabla^{i_2}_yh_j\del_{y_j}\nabla^{i_1}_yw_1+\nabla^{i_1}_yw_j\del_{y_j}\nabla_y^{i_2}h_1\Big),
%&\quad\quad+\nabla_y^{i_2}h_2\del_{y_2}\nabla_y^{i_1}w_1+\nabla_y^{i_1}w_2\del_{y_2}\nabla_y^{i_2}h_1+\nabla_y^{i_2}h_3\del_{y_3}\nabla_y^{i_1}w_1
%+\nabla_y^{i_1}w_3\del_{y_3}\nabla_y^{i_2}h_1\Big),
\endaligned
\ee
\bel{AAAE3-5RR2}
\aligned
\tilde{g}_3&:=\overline{T}^*e^{-\tau}\nabla_y^sg_2+(\overline{T}^*)^{{1\over2}}e^{-{1\over2}\tau}\sum_{i_1+i_2=s,~0\leq i_2\leq s}\sum_{j=1}^3\Big(\nabla_y^{i_1}h_j\del_{y_j}\nabla_y^{i_2}w_3-\nabla_y^{i_1}w_j\del_{y_j}\nabla_y^{i_2}q_3\Big)\\
&\quad-(\overline{T}^*)^{{1\over2}}e^{-{\tau\over2}}\sum_{i_1+i_2=s,~0\leq i_2\leq s}\sum_{j=1}^3\Big(\nabla^{i_1}_yh_j\del_{y_3}\nabla_y^{i_2}b_j-\nabla_y^{i_1}b_j\del_{y_3}\nabla_y^{i_2}h_j\Big).
%&\quad-(\overline{T}^*)^{{1\over2}}e^{-{1\over2}\tau}\sum_{i_1+i_2=s,~0\leq i_2\leq s-1}\sum_{j=1}^3
%\Big(\nabla^{i_2}_yh_j\del_{y_j}\nabla^{i_1}_yw_1+\nabla^{i_1}_yw_j\del_{y_j}\nabla_y^{i_2}h_1\Big),
%&\quad\quad+\nabla_y^{i_2}h_2\del_{y_2}\nabla_y^{i_1}w_1+\nabla_y^{i_1}w_2\del_{y_2}\nabla_y^{i_2}h_1+\nabla_y^{i_2}h_3\del_{y_3}\nabla_y^{i_1}w_1
%+\nabla_y^{i_1}w_3\del_{y_3}\nabla_y^{i_2}h_1\Big),
\endaligned
\ee

\begin{lemma}
Let viscosity constant $\nu$ and resistivity constant $\mu$ be sufficient big, constants $a\in(0,{1\over2}]$, $\bar{a},k\in(0,1]$ and $\forall s\in\NN^+$.
Assume that $\textbf{f},\textbf{g}\in \CC^1((0,+\infty),H^s(\Omega))$ and $(\textbf{w},\textbf{b})^T\in \Bcal_R$. The solution $(\textbf{h},\textbf{q})^T$ of linearized coupled system  (\ref{AE3-3-y-1})-(\ref{AENAY3-1}) with the initial data (\ref{E3-2R0}) and condition (\ref{E3-2R1}) satisfies
$$
\aligned
&\sum_{i=1}^3\int_{\Omega}\Big(|\nabla_y^sh_i|^2+|\nabla_y^sq_i|^2\Big)dy\\
&\lesssim e^{-C_{R,a,\bar{a},k,\nu,\mu}\tau}\sum_{i=1}^3\Big[\int_{\Omega}\Big(|\nabla_y^sh_{0i}|^2+|\nabla_y^sq_{0i}|^2\Big)dy+\int_0^{+\infty}\int_{\Omega}\Big(|\nabla_y^sf_i|^2+|\nabla_y^sg_i|^2\Big)dyd\tau\Big],
\endaligned
$$ 
where $C_{R,a,\bar{a},k,\nu,\mu}$ is a constant depending on constants $R,a,\bar{a},k,\nu,\mu$, and $R$ is a sufficient small constant.
\end{lemma}
\begin{proof}
We take the inner product to both sides of (\ref{AAE3-3-y-1})-(\ref{AAE3-5-y-3}) by $\nabla_y^sh_1$, $\nabla_y^sh_2$ and $\nabla_y^sh_3$, respectively, then integrating by parts, it holds
\bel{E3-6RR}
\aligned
&{1\over 2}{d\over d\tau}\|\nabla_y^sh_1\|_{\LL^2(\Omega)}^2+\nu\sum_{i,j=1}^3\|\del_{y_i} \nabla_y^sh_j\|^2_{\LL^2(\Omega)}+\Big(a(1+s)-{s\over2}+{3\over4}\Big)\|\nabla_y^sh_1\|^2_{\LL^2(\Omega)}
\\
&\quad+k(\overline{T}^*)^{2a+1}e^{-(2a+1)\tau}\int_{\Omega}\nabla_y^sh_1\cdot\nabla_y^sh_2dy
-3as\int_{\Omega}\Big(0,0,\del_{y_3}^sh_1\Big)^T\cdot\nabla_y^sh_1dy\\
%&\quad+(\overline{T}^*)^{{1\over2}}e^{-{1\over2}\tau}\sum_{i=1}^3\int_{\RR^3}\nabla_y^sh_1\cdot\Big(\nabla_y^sh_i\del_{y_i}w_1+w_i\del_{y_i}\nabla_y^sh_1\Big)dy\\
%+\nabla_y^sh_2\del_{y_2}w_1+w_2\del_{y_2}\nabla_y^sh_1\\
%&\quad\quad+\nabla_y^sh_3\del_{y_3}w_1+w_3\del_{y_3}\nabla_y^sh_1\Big)dy\\
&\quad+ks(\overline{T}^*)^{2a+1}e^{-(2a+1)\tau}\int_{\Omega}\Big(\del_{y_2}\del_{y_1}^{s-1}h_1,\del_{y_1}\del_{y_2}^{s-1}h_1,0\Big)^T\cdot\nabla_y^sh_1dy=\int_{\Omega}\nabla_y^sh_1\cdot\tilde{f}_1dy,
\endaligned
\ee
\bel{E3-7RR}
\aligned
&{1\over 2}{d\over d\tau}\|\nabla_y^sh_2\|_{\LL^2(\Omega)}^2+\nu\sum_{i,j=1}^3\|\del_{y_i} \nabla_y^sh_j\|^2_{\LL^2(\Omega)}+\Big(a(1+s)-{s\over2}+{3\over4}\Big)\|\nabla_y^sh_2\|^2_{\LL^2(\Omega)}
\\
&\quad-k(\overline{T}^*)^{2a+1}e^{-(2a+1)\tau}\int_{\Omega}\nabla_y^sh_1\cdot\nabla_y^sh_2dy
-3as\int_{\Omega}\Big(0,0,\del_{y_3}^sh_2\Big)^T\cdot\nabla_y^sh_2dy\\
%&\quad+(\overline{T}^*)^{{1\over2}}e^{-{1\over2}\tau}\sum_{i=1}^3\int_{\RR^3}\nabla_y^sh_1\cdot\Big(\nabla_y^sh_i\del_{y_i}w_1+w_i\del_{y_i}\nabla_y^sh_1\Big)dy\\
%+\nabla_y^sh_2\del_{y_2}w_1+w_2\del_{y_2}\nabla_y^sh_1\\
%&\quad\quad+\nabla_y^sh_3\del_{y_3}w_1+w_3\del_{y_3}\nabla_y^sh_1\Big)dy\\
&\quad+ks(\overline{T}^*)^{2a+1}e^{-(2a+1)\tau}\int_{\Omega}\Big(-\del_{y_2}\del_{y_1}^{s-1}h_2,\del_{y_1}\del_{y_2}^{s-1}h_2,0\Big)^T\cdot\nabla_y^sh_2dy=\int_{\Omega}\nabla_y^sh_2\cdot\tilde{f}_2dy,
\endaligned
\ee
and
\bel{E3-8RR}
\aligned
&{1\over 2}{d\over d\tau}\|\nabla_y^sh_3\|_{\LL^2(\Omega)}^2+\nu\sum_{i,j=1}^3\|\del_{y_i} \nabla_y^sh_j\|^2_{\LL^2(\Omega)}+\Big(a(s-2)-{s\over2}+{3\over4}\Big)\|\nabla_y^sh_3\|^2_{\LL^2(\Omega)}
\\
&\quad
-3as\int_{\Omega}\Big(0,0,\del_{y_3}^sh_3\Big)^T\cdot\nabla_y^sh_3dy
%&\quad+(\overline{T}^*)^{{1\over2}}e^{-{1\over2}\tau}\sum_{i=1}^3\int_{\RR^3}\nabla_y^sh_1\cdot\Big(\nabla_y^sh_i\del_{y_i}w_1+w_i\del_{y_i}\nabla_y^sh_1\Big)dy\\
%+\nabla_y^sh_2\del_{y_2}w_1+w_2\del_{y_2}\nabla_y^sh_1\\
%&\quad\quad+\nabla_y^sh_3\del_{y_3}w_1+w_3\del_{y_3}\nabla_y^sh_1\Big)dy\\
+ks(\overline{T}^*)^{2a+1}e^{-(2a+1)\tau}\int_{\Omega}\Big(-\del_{y_2}\del_{y_1}^{s-1}h_3,\del_{y_1}\del_{y_2}^{s-1}h_3,0\Big)^T\cdot\nabla_y^sh_3dy\\
&=\int_{\Omega}\nabla_y^sh_3\cdot\tilde{f}_3dy,
\endaligned
\ee

Similarly, we take the inner product to both sides of (\ref{AAENA3-3})-(\ref{AAENAY3-1}) by $\nabla_y^sq_1$, $\nabla_y^sq_2$ and $\nabla_y^sq_3$, respectively, then we integrate by parts to get 
\bel{AEa3-6RR}
\aligned
&{1\over 2}{d\over d\tau}\|\nabla_y^sq_1\|_{\LL^2(\Omega)}^2+\mu\sum_{i,j=1}^3\|\del_{y_i} \nabla_y^sq_j\|^2_{\LL^2(\Omega)}+\Big(a(s-1)-{s\over2}+{3\over4}\Big)\|\nabla_y^sq_1\|^2_{\LL^2(\Omega)}
\\
&\quad-3as\int_{\Omega}\Big(0,0,\del_{y_3}^sq_1\Big)^T\cdot\nabla_y^sq_1dy
+\bar{a}\overline{T}^*e^{-\tau}\int_{\Omega}\nabla_y^{s}h_1\cdot\nabla_y^sq_1dy\\
&\quad-\bar{a}\Big[\int_{\Omega}\Big(y_1\del_{y_1}\nabla_y^{s}h_1+y_2\del_{y_1}\nabla_y^{s}h_2-2y_3\del_{y_1}\nabla_y^{s}h_3\Big)\cdot\nabla_y^sq_1dy
+s\int_{\Omega}\Big(\del_{y_1}^sh_1,\del_{y_1}^sh_2,-2\del_{y_1}^sh_3\Big)^T\cdot\nabla_y^sq_1dy\Big]\\
&\quad-k(\overline{T}^*)^{2a+1}e^{-(2a+1)\tau}\Big[\int_{\Omega}\nabla_y^{s}h_2\cdot\nabla_y^sq_1dy
+s\int_{\Omega}\Big(\del_{y_2}\del_{y_1}^{s-1}q_1,\del_{y_1}\del_{y_2}^{s-1}q_1,0\Big)^T\cdot\nabla_y^sq_1dy\Big]\\
&\quad+{2\bar{a}k(\overline{T}^*)^{2a+{1\over2}}\over 4a+1}e^{-(2a+{1\over2})\tau}\Big[\int_{\Omega}\Big(y_1y_3\del_{y_1}\nabla_y^sh_2-y_2y_3\del_{y_1}\nabla_y^sh_1-\overline{T}^*e^{-\tau}y_3\nabla_y^sh_2\Big)\cdot\nabla_y^sq_1dy\\
&\quad+s\int_{\Omega}\Big(y_3\del_{y_1}^sh_2,-y_3\del_{y_1}\del_{y_2}^{s-1}h_1,y_1\del_{y_1}\del_{y_3}^{s-1}h_2-y_2\del_{y_1}\del_{y_3}^{s-1}h_1-\overline{T}^*e^{-\tau}\del_{y_3}^{s-1}h_2\Big)^T\cdot\nabla_y^sq_1dy\Big]
\\
&=\int_{\Omega}\nabla_y^sq_1\cdot\tilde{g}_1dy,
\endaligned
\ee
\bel{AEa3-7RR}
\aligned
&{1\over 2}{d\over d\tau}\|\nabla_y^sq_2\|_{\LL^2(\Omega)}^2+\mu\sum_{i,j=1}^3\|\del_{y_i} \nabla_y^sq_j\|^2_{\LL^2(\Omega)}+\Big(a(s-1)-{s\over2}+{3\over4}\Big)\|\nabla_y^sq_2\|^2_{\LL^2(\Omega)}
\\
&\quad-3as\int_{\Omega}\Big(0,0,\del_{y_3}^sq_2\Big)^T\cdot\nabla_y^sq_2dy
+\bar{a}\overline{T}^*e^{-\tau}\int_{\Omega}\nabla_y^{s}h_2\cdot\nabla_y^sq_2dy\\
&\quad-\bar{a}\Big[\int_{\Omega}\Big(y_1\del_{y_2}\nabla_y^{s}h_1+y_2\del_{y_2}\nabla_y^{s}h_2-2y_3\del_{y_2}\nabla_y^{s}h_3\Big)\cdot\nabla_y^sq_2dy
+s\int_{\Omega}\Big(\del_{y_2}^sh_1,\del_{y_2}^sh_2,-2\del_{y_2}^sh_3\Big)^T\cdot\nabla_y^sq_2dy\Big]\\
&\quad+k(\overline{T}^*)^{2a+1}e^{-(2a+1)\tau}\Big[\int_{\Omega}\nabla_y^{s}h_1\cdot\nabla_y^sq_2dy
+s\int_{\Omega}\Big(\del_{y_2}\del_{y_1}^{s-1}q_2,-\del_{y_1}\del_{y_2}^{s-1}q_2,0\Big)^T\cdot\nabla_y^sq_2dy\Big]\\
&\quad+{2\bar{a}k(\overline{T}^*)^{2a+{1\over2}}\over 4a+1}e^{-(2a+{1\over2})\tau}\Big[\int_{\Omega}\Big(y_1y_3\del_{y_2}\nabla_y^sh_2-y_2y_3\del_{y_2}\nabla_y^sh_1+\overline{T}^*e^{-\tau}y_3\nabla_y^sh_1\Big)\cdot\nabla_y^sq_2dy\\
&\quad+s\int_{\Omega}\Big(y_3\del_{y_2}\del_{y_1}^{s-1}h_2,-y_3\del_{y_2}^{s}h_1,y_1\del_{y_2}\del_{y_3}^{s-1}h_2-y_2\del_{y_2}\del_{y_3}^{s-1}h_1+\overline{T}^*e^{-\tau}\del_{y_3}^{s-1}h_1\Big)^T\cdot\nabla_y^sq_2dy\Big]
\\
&=\int_{\Omega}\nabla_y^sq_1\cdot\tilde{g}_1dy,
\endaligned
\ee
and
\bel{AEa3-8RR}
\aligned
&{1\over 2}{d\over d\tau}\|\nabla_y^sq_3\|_{\LL^2(\Omega)}^2+\mu\sum_{i,j=1}^3\|\del_{y_i} \nabla_y^sq_j\|^2_{\LL^2(\Omega)}+\Big(a(s+2)-{s\over2}+{3\over4}\Big)\|\nabla_y^sq_3\|^2_{\LL^2(\Omega)}
\\
&\quad-3as\int_{\Omega}\Big(0,0,\del_{y_3}^sq_3\Big)^T\cdot\nabla_y^sq_3dy
-2\bar{a}\overline{T}^*e^{-\tau}\int_{\Omega}\nabla_y^{s}h_3\cdot\nabla_y^sq_3dy\\
&\quad-\bar{a}\Big[\int_{\Omega}\Big(y_1\del_{y_3}\nabla_y^{s}h_1+y_2\del_{y_3}\nabla_y^{s}h_2-2y_3\del_{y_3}\nabla_y^{s}h_3\Big)\cdot\nabla_y^sq_3dy
+s\int_{\Omega}\Big(\del_{y_3}^sh_1,\del_{y_3}^sh_2,-2\del_{y_3}^sh_3\Big)^T\cdot\nabla_y^sq_3dy\Big]\\
&\quad+ks(\overline{T}^*)^{2a+1}e^{-(2a+1)\tau}
\int_{\Omega}\Big(-\del_{y_2}\del_{y_1}^{s-1}q_3,\del_{y_1}\del_{y_2}^{s-1}q_3,0\Big)^T\cdot\nabla_y^sq_3dy\\
&\quad+{2\bar{a}k(\overline{T}^*)^{2a+{1\over2}}\over 4a+1}e^{-(2a+{1\over2})\tau}\Big[\int_{\Omega}\Big(y_1y_3\del_{y_3}\nabla_y^sh_2-y_2y_3\del_{y_3}\nabla_y^sh_1+\overline{T}^*e^{-\tau}(y_2\nabla_y^sh_1-y_1\nabla_y^sh_2)\Big)\cdot\nabla_y^sq_3dy\\
&\quad+s\int_{\Omega}\Big(y_3\del_{y_3}\del_{y_1}^{s-1}h_2,-y_3\del_{y_3}\del_{y_2}^{s-1}h_1,\del_{y_3}^{s}h_2-y_2\del_{y_3}^{s}h_1+\overline{T}^*e^{-\tau}\del_{y_3}^{s-1}h_1\Big)^T\cdot\nabla_y^sq_3dy\\
&\quad+\overline{T}^*e^{-\tau}\Big(-\del_{y_1}^{s-1}h_2,\del_{y_2}^{s-1}h_1\Big)^T\cdot\nabla_y^sq_3\Big]
\\
&=\int_{\Omega}\nabla_y^sq_3\cdot\tilde{g}_3dy.
\endaligned
\ee

We sum up (\ref{E3-6RR})-(\ref{E3-8RR}) to get
\bel{E3-18}
\aligned
&{1\over 2}\sum_{i=1}^3{d\over d\tau}\Big(\|\nabla_y^sh_i\|_{\LL^2(\Omega)}^2+\|\nabla_y^sq_i\|_{\LL^2(\Omega)}^2\Big)+\sum_{i,j=1}^3\Big(3\nu\|\del_{y_i} \nabla_y^sh_j\|^2_{\LL^2(\Omega)}+3\mu\|\del_{y_i} \nabla_y^sq_j\|^2_{\LL^2(\Omega)}\Big)\\
&\quad+\Big(a(1+s)-{s\over2}+{3\over4}\Big)\sum_{i=1}^2\|\nabla_y^sh_i\|^2_{\LL^2(\Omega)}+\Big(a(s-2)-{s\over2}+{3\over4}\Big)\|\nabla_y^sh_3\|^2_{\LL^2(\Omega)}\\
&\quad+\Big(a(s-1)-{s\over2}+{3\over4}\Big)\sum_{i=1}^2\|\nabla_y^sq_i\|^2_{\LL^2(\Omega)}+\Big(a(s+2)-{s\over2}+{3\over4}\Big)\|\nabla_y^sq_3\|^2_{\LL^2(\Omega)}\\
&\quad+I_1+I_2+I_3
=\sum_{i=1}^3\int_{\Omega}\Big(\nabla_y^sh_i\cdot\tilde{f}_i+\nabla_y^sq_i\cdot\tilde{g}_i\Big)dy,
\endaligned
\ee
where coupled terms take the following form
\bel{YAA1-1}
\aligned
I_1&:=-3as\sum_{i=1}^3\int_{\Omega}\Big(0,0,\del_{y_3}^sh_i\Big)^T\cdot\nabla_y^sh_idy\\
&\quad+ks(\overline{T}^*)^{2a+1}e^{-(2a+1)\tau}\sum_{i=1}^3\int_{\Omega}\Big(\del_{y_2}\del_{y_1}^{s-1}h_i,\del_{y_1}\del_{y_2}^{s-1}h_i,0\Big)^T\cdot\nabla_y^sh_idy\\
&\quad-\bar{a}\sum_{i=1}^3\Big[\int_{\Omega}\Big(y_1\del_{y_i}\nabla_y^{s}h_1+y_2\del_{y_i}\nabla_y^{s}h_2-2y_3\del_{y_i}\nabla_y^{s}h_3\Big)\cdot\nabla_y^sq_idy\\
&\quad+s\int_{\Omega}\Big(\del_{y_i}^sh_1,\del_{y_i}^sh_2,-2\del_{y_i}^sh_3\Big)^T\cdot\nabla_y^sq_idy\Big].
\endaligned
\ee

\bel{YAA1-2}
\aligned
I_2&:=-k(\overline{T}^*)^{2a+1}e^{-(2a+1)\tau}\Big[\int_{\Omega}\Big(\nabla_y^{s}h_2\cdot\nabla_y^sq_1dy
-\nabla_y^{s}h_1\cdot\nabla_y^sq_2\Big)dy\\
&\quad+s\int_{\Omega}\Big(\del_{y_2}\del_{y_1}^{s-1}q_1,\del_{y_1}\del_{y_2}^{s-1}q_1,0\Big)^T\cdot\nabla_y^sq_1dy
-s\int_{\Omega}\Big(\del_{y_2}\del_{y_1}^{s-1}q_2,-\del_{y_1}\del_{y_2}^{s-1}q_2,0\Big)^T\cdot\nabla_y^sq_2dy\\
&\quad+s\int_{\Omega}\Big(\del_{y_3}^sh_1,\del_{y_3}^sh_2,-2\del_{y_3}^sh_3\Big)^T\cdot\nabla_y^sq_3dy\Big],
\endaligned
\ee

\bel{YAA1-3}
\aligned
I_3&:={2\bar{a}k(\overline{T}^*)^{2a+{1\over2}}\over 4a+1}e^{-(2a+{1\over2})\tau}\Big[\sum_{i=1}^3\int_{\Omega}\Big(y_1y_3\del_{y_i}\nabla_y^sh_2-y_2y_3\del_{y_i}\nabla_y^sh_1\Big)\cdot\nabla_y^sq_idy\\
&\quad+\overline{T}^*e^{-{1\over2}\tau}\int_{\Omega}\Big(-y_3\nabla_y^sh_2\cdot\nabla_y^sq_1+y_3\nabla_y^sh_1\cdot\nabla_y^sq_2
+(y_2\nabla_y^sh_1-y_1\nabla_y^sh_2)\cdot\nabla_y^sq_3\Big)dy\\
&\quad+\int_{\Omega}\Big(y_3\del_{y_1}^sh_2,-y_3\del_{y_1}\del_{y_2}^{s-1}h_1,y_1\del_{y_1}\del_{y_3}^{s-1}h_2-y_2\del_{y_1}\del_{y_3}^{s-1}h_1-\overline{T}^*e^{-\tau}\del_{y_3}^{s-1}h_2\Big)^T\cdot\nabla_y^sq_1dy\\
&\quad+s\int_{\Omega}\Big(y_3\del_{y_2}\del_{y_1}^{s-1}h_2,-y_3\del_{y_2}^{s}h_1,y_1\del_{y_2}\del_{y_3}^{s-1}h_2-y_2\del_{y_2}\del_{y_3}^{s-1}h_1+\overline{T}^*e^{-\tau}\del_{y_3}^{s-1}h_1\Big)^T\cdot\nabla_y^sq_2dy\\
&\quad+s\int_{\Omega}\Big(y_3\del_{y_3}\del_{y_1}^{s-1}h_2,-y_3\del_{y_3}\del_{y_2}^{s-1}h_1,\del_{y_3}^{s}h_2-y_2\del_{y_3}^{s}h_1+\overline{T}^*e^{-\tau}\del_{y_3}^{s-1}h_1\Big)^T\cdot\nabla_y^sq_3dy\\
&\quad+\overline{T}^*e^{-\tau}\Big(-\del_{y_1}^{s-1}h_2,\del_{y_2}^{s-1}h_1\Big)^T\cdot\nabla_y^sq_3\Big].
\endaligned
\ee

We now estimate each coupled term in (\ref{E3-18}). On one hand,
note that $y\in\Omega=([0,1])^3$ and $(\textbf{w},\textbf{b})^T\in \Bcal_R$. We employ Young's inequality and Poincar\'{e} inequality, and integrating by parts to derive
\bel{E3-19}
\aligned
&\Big|I_1\Big|\lesssim C_{a,k,\bar{a}}\sum_{i=1}^3\int_{\Omega}\Big(|\nabla_y^sh_i|^2+|\del_{y_i}\nabla_y^{s}h_i|^2+|\del_{y_i}\nabla_y^{s-1}h_i|^2+|\nabla_y^sq_i|^2\Big)dy,\\
&\quad~~\lesssim C_{a,k,\bar{a}}\sum_{i=1}^3\int_{\Omega}\Big(|\nabla_y^sh_i|^2+|\del_{y_i}\nabla_y^{s}h_i|^2+|\nabla_y^sq_i|^2\Big)dy,\\
&\Big|I_2\Big|\lesssim C_{k}\Big[\sum_{i=1}^3\int_{\Omega}\Big(|\nabla_y^sh_i|^2+|\nabla_y^sq_i|^2\Big)dy
+\sum_{j=1}^2\int_{\Omega}|\del_{y_j}\nabla_y^{s}q_j|^2dy\Big],\\
&\Big|I_3\Big|\lesssim C_{a,k,\bar{a}}\Big[\sum_{i=1}^3\int_{\Omega}\Big(|\nabla_y^sh_i|^2+|\nabla_y^sq_i|^2\Big)dy
+\sum_{i=1}^3\sum_{j=1}^2\int_{\Omega}|\del_{y_i}\nabla_y^{s}h_j|^2dy\Big],
\endaligned
\ee
where $C_{a,k,\bar{a}}$ and $C_{k}$ denote postive constants depending on $a$, $k$ and $\bar{a}$.

On the other hand, by (\ref{AE3-5RR0})-(\ref{AE3-5RR2}), we know the highest order derivatives on $h_i$ of $\del_{y_i}\nabla_y^s\overline{f}$ is $s$. So
 we can use the standard Calder\'on-Zygmund theory, Young's inequality,  $H^{s}(\Omega)\subset L^{\infty}(\Omega)$ $(s\geq2)$ and integrating by parts to derive
\bel{E3-22}
\Big|\sum_{i=1}^3\int_{\Omega}\nabla_y^sh_i\cdot\del_{y_i}\nabla_y^s\overline{f}dy\Big|\lesssim C_{R}\sum_{i=1}^3\|\nabla^s_yh_i\|^2_{\LL^2(\Omega)},
\ee
furthermore, by (\ref{AE3-5RR0})-(\ref{AAAE3-5RR0}) and Poincar\'{e} inequality, we derive
\bel{E3-22R-0}
\Big|\sum_{i=1}^3\int_{\Omega}\Big(\nabla_y^sh_i\cdot\tilde{f}_i\Big)dy\Big|\lesssim (C_{R}+{a\over2})\sum_{i=1}^3\|\nabla^s_yh_i\|^2_{\LL^2(\Omega)}
+2a^{-1}\sum_{i=1}^3\|\nabla_y^{s}f_i\|_{\LL^2(\Omega)}^2,
\ee
where $C_{R}$ is a postive constant depending on $R$, which is small constant as $R$ small.

Similarly to get (\ref{E3-22R-0}), by (\ref{AAAE3-5RR0})-(\ref{AAAE3-5RR2}), we can also use Poincar\'{e} inequality and Young inequality to derive
\bel{E3-22R-11}
\Big|\sum_{i=1}^3\int_{\Omega}\Big(\nabla_y^sq_i\cdot\tilde{g}_i\Big)dy\Big|\lesssim (C_{R}+{a\over2})\sum_{i=1}^3\|\nabla^s_yq_i\|^2_{\LL^2(\Omega)}
+2a^{-1}\sum_{i=1}^3\|\nabla_y^{s}g_i\|_{\LL^2(\Omega)}^2.
\ee

Hence we can apply estimates (\ref{E3-19})-(\ref{E3-22R-11}) to (\ref{E3-18}), it holds
\bel{E3-23}
\aligned
&{1\over 2}\sum_{i=1}^3{d\over d\tau}\Big(\|\nabla_y^sh_i\|_{\LL^2(\Omega)}^2+\|\nabla_y^sq_i\|_{\LL^2(\Omega)}^2\Big)+\Big(3\nu-C_{a,k,\bar{a}}\Big)\sum_{i,j=1}^3\|\del_{y_i} \nabla_y^sh_j\|^2_{\LL^2(\Omega)}\\
&\quad+\Big(3\mu-C_{a,k,\bar{a}}\Big)\sum_{i,j=1}^3\|\del_{y_i} \nabla_y^sq_j\|^2_{\LL^2(\Omega)}+\Big(a({1\over2}+s)-{s\over2}+{3\over4}-C_{a,k,\bar{a},R}\Big)\sum_{i=1}^2\|\nabla_y^sh_i\|^2_{\LL^2(\Omega)}\\
&\quad+\Big(a(s-{5\over2})-{s\over2}+{3\over4}-C_{a,k,\bar{a},R}\Big)\|\nabla_y^sh_3\|^2_{\LL^2(\Omega)}+\Big(a(s-{3\over2})-{s\over2}+{3\over4}-C_{a,k,\bar{a},R}\Big)\sum_{i=1}^2\|\nabla_y^sq_i\|^2_{\LL^2(\Omega)}\\
&\quad+\Big(a(s+{3\over2})-{s\over2}+{3\over4}-C_{a,k,\bar{a},R}\Big)\|\nabla_y^sq_3\|^2_{\LL^2(\Omega)}\\
&\lesssim 2a^{-1}\sum_{i=1}^3\Big(\|\nabla_y^{s}f_i\|_{\LL^2(\Omega)}^2+\|\nabla_y^{s}g_i\|_{\LL^2(\Omega)}^2\Big),
\endaligned
\ee
where $C_{a,k,\bar{a},R}$ is a postive constant depending on constants $a,k,\bar{a},R$.

Furthermore, by Poincar\'{e} inequality, it holds
$$
\aligned
&\|\del_{y_i} \nabla_y^sh_j\|^2_{\LL^2(\Omega)}\gtrsim \| \nabla_y^sh_j\|^2_{\LL^2(\Omega)},\\
&\|\del_{y_i} \nabla_y^sq_j\|^2_{\LL^2(\Omega)}\gtrsim \| \nabla_y^sq_j\|^2_{\LL^2(\Omega)},
\endaligned
$$
which combining with (\ref{E3-23}) gives that
\bel{AAE3-23}
\aligned
{1\over 2}\sum_{i=1}^3{d\over d\tau}\Big(\|\nabla_y^sh_i\|_{\LL^2(\Omega)}^2+\|\nabla_y^sq_i\|_{\LL^2(\Omega)}^2\Big)
&+\Big(3\nu+{a\over2}+(a-{1\over2})s+{3\over4}-C_{a,k,\bar{a},R}\Big)\sum_{i=1}^2\|\nabla_y^sh_i\|^2_{\LL^2(\Omega)}\\
&\quad+\Big(3\nu-{5a\over2}+(a-{1\over2})s+{3\over4}-C_{a,k,\bar{a},R}\Big)\|\nabla_y^sh_3\|^2_{\LL^2(\Omega)}\\
&\quad+\Big(3\mu-{3a\over2}+(a-{1\over2})s+{3\over4}-C_{a,k,\bar{a},R}\Big)\sum_{i=1}^2\|\nabla_y^sq_i\|^2_{\LL^2(\Omega)}\\
&\quad+\Big(3\mu+{3a\over2}+(a-{1\over2})s+{3\over4}-C_{a,k,\bar{a},R}\Big)\|\nabla_y^sq_3\|^2_{\LL^2(\Omega)}\\
&\lesssim 2a^{-1}\sum_{i=1}^3\Big(\|\nabla_y^{s}f_i\|_{\LL^2(\Omega)}^2+\|\nabla_y^{s}g_i\|_{\LL^2(\Omega)}^2\Big),
\endaligned
\ee

It is easy to see that for sufficient small constant $R$, $\bar{a},k\in(0,1]$ and sufficient big constants $\mu,\nu$, we can choose a fixed constant $a\in(0,{1\over2})$ and any fixed integer $s$, it holds
$$
\aligned
&3\nu+{a\over2}+(a-{1\over2})s+{3\over4}-C_{a,k,\bar{a},R}>0,\\
&3\nu-{5a\over2}+(a-{1\over2})s+{3\over4}-C_{a,k,\bar{a},R}>0,\\
&3\mu-{3a\over2}+(a-{1\over2})s+{3\over4}-C_{a,k,\bar{a},R}>0,\\
&3\mu+{3a\over2}+(a-{1\over2})s+{3\over4}-C_{a,k,\bar{a},R}>0.
\endaligned
$$

Therefore, applying Gronwall's inequality to (\ref{AAE3-23}), there exists a positive constant $C_{R,a,\bar{a},k,\nu,\mu}$ depending on $R,a,\bar{a},k,\nu,\mu$ such that
$$
\aligned
\sum_{i=1}^3\Big(\|\nabla_yh_i\|_{\LL^2(\Omega)}^2+\|\nabla_yq_i\|_{\LL^2(\Omega)}^2\Big)&\lesssim 
e^{-C_{R,a,\bar{a},k,\nu,\mu}\tau}\sum_{i=1}^3\Big[\|\nabla_yh_i(0,(\overline{T}^*)^{{1\over2}}y)\|_{\LL^2(\Omega)}^2
+|\nabla_yq_i(0,(\overline{T}^*)^{{1\over2}}y)\|_{\LL^2(\Omega)}^2\\
&\quad+\int_0^{+\infty}\Big(\|\nabla_yf_i\|^2_{\LL^2(\Omega)}+\|\nabla_yg_i\|^2_{\LL^2(\Omega)}\Big)d\tau\Big],\quad \forall \tau>0.
\endaligned
$$

\end{proof}

Similar to get the estimate in Lemma 3.2, we apply the operator $\del_{\tau}$ to both sides of (\ref{AAE3-3-y-1})-(\ref{AAENAY3-1}), then using the similar process of proof in Lemma 3.2 , we can obtain the following result. Here we omit the details.

\begin{lemma}
Let viscosity constant $\nu$ and resistivity constant $\mu$ be sufficient big, constants $a\in(0,{1\over2}]$, $\bar{a},k\in(0,1]$ and $\forall s\in\NN^+$.
Assume that $\textbf{f},\textbf{g}\in \CC^1((0,+\infty),H^s(\Omega))$ and $(\textbf{w},\textbf{b})^T\in \Bcal_R$. The solution $(\textbf{h},\textbf{q})^T$ of linearized coupled system  (\ref{AE3-3-y-1})-(\ref{AENAY3-1}) with the initial data (\ref{E3-2R0}) and condition (\ref{E3-2R1}) satisfies
$$
\aligned
&\sum_{i=1}^3\int_{\Omega}\Big(|\nabla_y^s\del_{\tau}h_i|^2+|\nabla_y^s\del_{\tau}q_i|^2\Big)dy\\
&\lesssim e^{-C_{R,a,\bar{a},k,\nu,\mu}\tau}\sum_{i=1}^3\Big[\int_{\Omega}\Big(|\nabla_y^s\del_{\tau}h_{0i}|^2+|\nabla_y^s\del_{\tau}q_{0i}|^2\Big)dy+\int_0^{+\infty}\int_{\Omega}\Big(|\nabla_y^s\del_{\tau}f_i|^2+|\nabla_y^s\del_{\tau}g_i|^2\Big)dyd\tau\Big],
\endaligned
$$ 
where $C_{R,a,\bar{a},k,\nu,\mu}$ is a positive constant depending on constants $R,a,\bar{a},k,\nu,\mu$, and $R$ is a sufficient small constant.
\end{lemma}

\subsection{Global existence of solutions for the linear system in self-similarity coordinates}

\begin{proposition}

Let viscosity constant $\nu$ and resistivity constant $\mu$ be sufficient big, constants $a\in(0,{1\over2}]$, $\bar{a},k\in(0,1]$ and $\forall s\in\NN^+$.
Assume that $\textbf{f},\textbf{g}\in \CC^1((0,+\infty),H^s(\Omega))$ and $(\textbf{w},\textbf{b})^T\in \Bcal_R$. Then
the linear problem (\ref{E3-2})-(\ref{E3-2rr1}) with the initial data (\ref{E3-2R0}) and boundary condition (\ref{E3-2R1})
admits a solution 
$$
(\textbf{h}(\tau,y),\textbf{q}(\tau,y))\in\bigcap_{i= 0}^1\CC^i([0,+\infty);H^{s-i}(\Omega)\times H^{s-i}(\Omega)),
$$
which satisfies
\bel{AE4-10RX}
\|\textbf{h}^{(1)}\|^2_{H^s(\Omega)}+\|\textbf{q}^{(1)}\|^2_{H^s(\Omega)}\lesssim e^{-C_{R,a,\bar{a},k,\nu,\mu}\tau}\Big(\|\textbf{h}_0\|_{H^s(\Omega)}^2+\|\textbf{q}_0\|_{H^s(\Omega)}^2+ \|\textbf{f}\|^2_{H^s(\Omega)}+ \|\textbf{g}\|^2_{H^s(\Omega)}\Big),
\ee
and 
\bel{AE4-10R1}
\|\textbf{h}^{(1)}\|^2_{\Ccal_1^s}+\|\textbf{q}^{(1)}\|^2_{C_1^s}\lesssim\|\textbf{h}_0\|_{\Ccal_1^s}^2+\|\textbf{q}_0\|_{C_1^s}^2+ \|\textbf{f}\|^2_{\Ccal_1^s}+ \|\textbf{g}\|^2_{C_1^s},
\ee
where $C_{R,a,\bar{a},k,\nu,\mu}$ is a positive constant depending on constants $R,a,\bar{a},k,\nu,\mu$, and $R$ is a sufficient small constant.
\end{proposition}
\begin{proof}
Let $\PP$ the Leray projector onto the space of divergence free functions.
We apply the Leray projector to (\ref{E3-2})-(\ref{E3-2rr1}), it holds
\bel{E3-26}
\aligned
&\textbf{h}_t-\nu\PP\triangle\textbf{h}+\NN_1(\textbf{w},\textbf{b},\textbf{h},\textbf{q})=\PP \textbf{f},\\
&\textbf{q}_t-\mu\PP\triangle\textbf{q}+\NN_2(\textbf{w},\textbf{b},\textbf{h},\textbf{q})=\PP\textbf{g},
\endaligned
\ee
where
$$
\aligned
&\NN_1(\textbf{w},\textbf{b},\textbf{h},\textbf{q})=\PP\Big(\textbf{h}\cdot\nabla(\overline{\textbf{v}}_{\overline{T}^*}+\textbf{w})+(\overline{\textbf{v}}_{\overline{T}^*}+\textbf{w})\cdot\nabla\textbf{h}-(\overline{\textbf{H}}_{\overline{T}^*}+\textbf{b})\cdot\nabla\textbf{q}\\
&\quad\quad\quad\quad\quad\quad\quad-\textbf{q}\cdot\nabla(\overline{\textbf{H}}_{\overline{T}^*}+\textbf{b})+\nabla((\overline{\textbf{H}}_{\overline{T}^*}+\textbf{b})\cdot\textbf{q})\Big),\\
&\NN_2(\textbf{w},\textbf{b},\textbf{h},\textbf{q})=\PP\Big(-(\overline{\textbf{H}}_{\overline{T}^*}+\textbf{b})\cdot\nabla\textbf{h}-\textbf{q}\cdot\nabla(\overline{\textbf{v}}_{\overline{T}^*}+\textbf{w})+(\overline{\textbf{v}}_{\overline{T}^*}+\textbf{w})\cdot\nabla\textbf{q}
+\textbf{h}\cdot\nabla(\overline{\textbf{H}}_{\overline{T}^*}+\textbf{b})\Big).
\endaligned
$$

We recall that equations (\ref{E3-26}) can be rewritten as linear coupled system (\ref{AE3-3-y-1})-(\ref{AENAY3-1}) in the similarity coordinates (\ref{Y-1}) by applying the Leray projector $\PP$ to them.
For convenience, we write them as 
\bel{E3-26x}
{\del\over\del\tau}\left(
\begin{array}{cccc}
\textbf{h}\\
\textbf{q}
\end{array}
\right)
+
\left(
\begin{array}{cccc}
-\nu\PP\triangle_y&0\\
0&-\mu\PP\triangle_y
\end{array}
\right)
\left(
\begin{array}{cccc}
\textbf{h}\\
\textbf{q}
\end{array}
\right)
+
\left(
\begin{array}{cccc}
\overline{\NN}_1(\textbf{w},\textbf{b},\textbf{h},\textbf{q})\\
\overline{\NN}_2(\textbf{w},\textbf{b},\textbf{h},\textbf{q})
\end{array}
\right)
=
\left(
\begin{array}{cccc}
\PP\textbf{f}\\
 \PP\textbf{g}
\end{array}
\right).
\ee

We notice that there is no singular coefficient in (\ref{E3-26x}) for $\tau\in(0,+\infty)$,
and the term
$BP:=\left(
\begin{array}{cccc}
\overline{\NN}_1(\textbf{w},\textbf{b},\circ,\circ)\\
\overline{\NN}_2(\textbf{w},\textbf{b},\circ,\circ)
\end{array}
\right)$ can be seen as a bounded perturbation of the linear operator
$ML:=\left(
\begin{array}{cccc}
-\nu\PP\triangle_y&0\\
0&-\mu\PP\triangle_y
\end{array}
\right)$.

Thus the linear operator
$$
\ZZ:=ML+BP
$$
can generate a strongly continuous semigroup $e^{\ZZ \tau}$ in Sobolev space $H^s\times H^s$ (see \cite{Yan3,Yu}).
Hence linear equations (\ref{E3-26x}) has a solution in $\bigcap_{i= 0}^1\CC^i([0,+\infty);H^{s-i}(\Omega)\times H^{s-i}(\Omega))$.
Then, from Lemma 3.3-3.4, we can get (\ref{AE4-10R1}) holds.
\end{proof}

\subsection{Local existence of solutions for the linear system in original coordinates}

Recall the self-similarity coordinates (\ref{Y-1}), the original coordinate can be expressed by the self-similarity coordinates as follows
$$
t=T(1-e^{-\tau}),\quad
x=y\sqrt{\overline{T}^*-t},
$$
so we can directly use Proposition 3.1 to get the following result.

\begin{proposition}
Let viscosity constant $\nu$ and resistivity constant $\mu$ be sufficient big, constants $a\in(0,{1\over2}]$, $\bar{a},k\in(0,1]$ and $\forall s\in\NN^+$.
Assume that $\textbf{f},\textbf{g}\in \CC^1((0,\overline{T}^*),H^s(\Omega_t))$ and $(\textbf{w},\textbf{b})^T\in \Bcal_R$.
The linearized system (\ref{E3-2})-(\ref{E3-2rrx}) with the initial data (\ref{AAAAE3-2R0}) and condition (\ref{AAAAE3-2R1}) 
admits a solution
$$
(\textbf{h}(t,x),\textbf{q}(t,x))^T\in\Ccal^s_{1}:=\bigcap_{i= 0}^1\CC^i((0,\overline{T}^*);H^{s-i}(\Omega)\times H^{s-i}(\Omega)).
$$
Moreover, it satisfies
\bel{E3-25R1X}
\|\textbf{h}^{(1)}\|^2_{H^s(\Omega_t)}+\|\textbf{q}^{(1)}\|^2_{H^s(\Omega_t)}\lesssim (\overline{T}^*-t)^{C_{R,a,\bar{a},k,\nu,\mu}}\Big(\|\textbf{h}_0\|_{H^s(\Omega_t)}^2+\|\textbf{q}_0\|_{H^s(\Omega_t)}^2+ \|\textbf{f}\|^2_{H^s(\Omega_t)}+ \|\textbf{g}\|^2_{H^s(\Omega_t)}\Big),
\ee
and 
\bel{E3-25R1}
\|\textbf{h}(t,x)\|^2_{\Ccal_1^s}+\|\textbf{q}(t,x)\|^2_{\Ccal_1^s}
\lesssim \|\textbf{h}_0\|^2_{\Ccal_1^s}+\|\textbf{q}_0\|^2_{\Ccal_1^s}+\|\textbf{f}\|^2_{\Ccal_1^s}+\|\textbf{g}\|^2_{\Ccal_1^s},\quad \forall t\in(0,\overline{T}^*),
\ee
where $C_{R,a,\bar{a},k,\nu,\mu}$ is a positive constant depending on constants $R,a,\bar{a},k,\nu,\mu$, and $R$ is a sufficient small constant.
\end{proposition}

%--------------------------------------------------------------------------------------------------------

\section{Asymptotic stability of explicit blowup solutions}
\setcounter{equation}{0}

The stability of the explicit blowup solutions $(\overline{\textbf{v}}_{\overline{T}^*},\overline{\textbf{H}}_{\overline{T}^*})$
is equivalent to prove the local existence of solutions $(\textbf{w}(t,x),\textbf{b}(t,x))$ for equations (\ref{E3-1}) with a given small initial data.
Meanwhile, this solution $(\textbf{w}(t,x),\textbf{b}(t,x))$ should be sufficient small in some Sobolev space as the time $t$ goes to the blowup time $\overline{T}^*$.

\subsection{The approximation scheme}

We will construct a local higher regular solution for nonlinear equations (\ref{E3-1}) by using a suitable new Nash-Moser iteration scheme, which has been used in \cite{Yan,Yan3}.
We introduce a family of smooth operators possessing the following properties.

\begin{lemma}\cite{Alin,H1}
There is a family $\{\Pi_{\theta}\}_{\theta\geq1}$ of smoothing operators in the space $H^s(\Omega)$ acting on the class of functions such that
\bel{AH1}
\aligned
&\|\Pi_{\theta}\textbf{w}\|_{H^{k_1}}\leq C\theta^{(k_1-k_2)_+}\|\textbf{w}\|_{H^{k_2}},~~k_1,~k_2\geq0,\\
&\|\Pi_{\theta}\textbf{w}-\textbf{w}\|_{H^{k_1}}\leq C\theta^{k_1-k_2}\|\textbf{w}\|_{H^{k_2}},~~0\leq k_1\leq k_2,\\
&\|\frac{d}{d\theta}\Pi_{\theta}\textbf{w}\|_{H^{k_1}}\leq C\theta^{(k_1-k_2)_+-1}\|\textbf{w}\|_{H^{k_2}},~~k_1,~k_2\geq0,
\endaligned
\ee
where $C$ is a positive constant and $(s_1-s_2)_+:=\max(0,s_1-s_2)$.
\end{lemma}

In our iteration scheme, we set
$$
\theta=N_m=2^m,\quad \forall m= 0,1,2,\ldots.
$$
Then, by (\ref{AH1}), there is
\bel{AH2}
\|\Pi_{N_m}\textbf{w}\|_{H^{k_1}}\lesssim N_m^{k_1-k_2}\|\textbf{w}\|_{H^{k_2}},\quad \forall k_1\geq k_2.
\ee

We consider the approximation problem of nonlinear equations (\ref{E3-1}) as follows

\bel{E4-1}
\aligned
\Lcal(\textbf{w},\textbf{b}):=&\textbf{w}_t-\nu\triangle\textbf{w}+\textbf{w}\cdot\nabla\overline{\textbf{v}}_{\overline{T}^*}+\overline{\textbf{v}}_{\overline{T}^*}\cdot\nabla\textbf{w}
-\overline{\textbf{H}}_{\overline{T}^*}\cdot\nabla\textbf{b}-\textbf{b}\cdot\nabla\overline{\textbf{H}}_{\overline{T}^*}+\nabla(\overline{\textbf{H}}_{\overline{T}^*}\cdot\textbf{b})\\
&\quad+\Pi_{N_m}\Big(\textbf{w}\cdot\nabla\textbf{w}-\nabla p-\textbf{b}\cdot\nabla\textbf{b}+\nabla({|\textbf{b}|^2\over 2})\Big),\\
\Jcal(\textbf{w},\textbf{b}):=&\textbf{b}_t-\mu\triangle\textbf{b}-\overline{\textbf{H}}_{\overline{T}^*}\cdot\nabla\textbf{w}-\textbf{b}\cdot\nabla\overline{\textbf{v}}_{\overline{T}^*}+\overline{\textbf{v}}_{\overline{T}^*}\cdot\nabla\textbf{b}+\textbf{w}\cdot\nabla\overline{\textbf{H}}_{\overline{T}^*}\\
&\quad+\Pi_{N_m}\Big(-\textbf{b}\cdot\nabla\textbf{w}+\textbf{w}\cdot\nabla\textbf{b}\Big),\\
\nabla\cdot\textbf{w}=&0,\quad \nabla\cdot\textbf{b}=0,
\endaligned
\ee
with initial data
$$
\textbf{w}(0,x)=\textbf{w}_0(x)\in H^s(\Omega_t),\quad \textbf{b}(0,x)=\textbf{b}_0(x)\in H^s(\Omega_t),
$$
and boundary conditions
$$
\textbf{w}(t,x)|_{x\in\del\Omega_t}=0,\quad \textbf{b}(t,x)|_{x\in\del\Omega_t}=0,
$$
where $(t,x)\in(0,\overline{T}^*)\times\Omega_t$,  the pressure $p$ is given in (\ref{E3-0R1}),  $\nabla\overline{\textbf{v}}_{\overline{T}^*}$ and $\nabla\overline{\textbf{H}}_{\overline{T}^*}$ are given in (\ref{xx1-1})-(\ref{xx1-2}), respectively.

Assume that the $m$-th approximation solutions of (\ref{E4-1}) is denoted by $(\textbf{w}^{(m)},\textbf{b}^{(m)})$ with $m=0,1,2,\ldots$. Let
$$
\aligned
&\textbf{h}^{(m)}:=\textbf{w}^{(m)}-\textbf{w}^{(m-1)},\\
&\textbf{q}^{(m)}:=\textbf{b}^{(m)}-\textbf{b}^{(m-1)},
\endaligned
$$
then we have
$$
\aligned
&\textbf{w}^{(m)}=\textbf{w}^{(0)}+\sum_{i=1}^m\textbf{h}^{(i)},\\
&\textbf{b}^{(m)}=\textbf{b}^{(0)}+\sum_{i=1}^m\textbf{q}^{(i)}.
\endaligned
$$
Our target is to prove that $(\textbf{w}^{(\infty)},\textbf{b}^{(\infty)})$ is a local solution of nonlinear equations (\ref{E3-1}). It is equivalent to show the series $\sum\limits_{i=1}^m\textbf{h}^{(i)}$
and $\sum\limits_{i=1}^m\textbf{q}^{(i)}$  are convergence.

We linearize nonlinear equations (\ref{E3-1}) around $(\textbf{w}^{(m-1)},\textbf{b}^{(m-1)})$ to get the linearized operators as follows
$$
\aligned
\Lcal[(\textbf{w}^{(m-1)},\textbf{b}^{(m-1)})](\textbf{h}^{(m)},\textbf{q}^{(m)}):=&\textbf{h}^{(m)}_t-\nu\triangle\textbf{h}^{(m)}+\textbf{h}^{(m)}\cdot\nabla(\overline{\textbf{v}}_{\overline{T}^*}+\textbf{w}^{(m-1)})
+(\overline{\textbf{v}}_{\overline{T}^*}+\textbf{w}^{(m-1)})\cdot\nabla\textbf{h}^{(m)}
\\
&\quad-\nabla[(\Fcal_{\textbf{w}^{(m-1)}} p)\textbf{h}^{(m)}+(\Fcal_{\textbf{b}^{(m-1)}} p)\textbf{q}^{(m)}]-(\overline{\textbf{H}}_{\overline{T}^*}+\textbf{b}^{(m-1)})\cdot\nabla\textbf{q}^{(m)}\\
&\quad-\textbf{q}^{(m)}\cdot\nabla(\overline{\textbf{H}}_{\overline{T}^*}+\textbf{b}^{(m-1)})+\nabla((\overline{\textbf{H}}_{\overline{T}^*}+\textbf{b}^{(m-1)})\cdot\textbf{q}^{(m)}),
\endaligned
$$
$$
\aligned
\Jcal[(\textbf{w}^{(m-1)},\textbf{b}^{(m-1)})](\textbf{h}^{(m)},\textbf{q}^{(m)}):=&\textbf{q}^{(m)}_t-\mu\triangle\textbf{q}^{(m)}-(\overline{\textbf{H}}_{\overline{T}^*}+\textbf{b}^{(m-1)})\cdot\nabla\textbf{h}^{(m)}-\textbf{q}^{(m)}\cdot\nabla(\overline{\textbf{v}}_{\overline{T}^*}+\textbf{w}^{(m-1)})\\
&\quad+(\overline{\textbf{v}}_{\overline{T}^*}+\textbf{w}^{(m-1)})\cdot\nabla\textbf{q}^{(m)}
+\textbf{h}^{(m)}\cdot\nabla(\overline{\textbf{H}}_{\overline{T}^*}+\textbf{b}^{(m-1)}).
\endaligned
$$

Let constants $k_0>2$ and $0<2\eps_0<\eps\ll1$.
We choose the approximation function $(\textbf{w}^{(0)},\textbf{b}^{(0)})\in \Ccal_1^{k_0+3}\times\Ccal_1^{k_0+3}$ satisfying
\bel{E4-5}
\aligned
&\textbf{w}^{(0)}\neq(0,0,0)^T,\quad \textbf{b}^{(0)}\neq(0,0,0)^T,\\
&\nabla\cdot\textbf{w}^{(0)}=0,\quad \nabla\cdot\textbf{b}^{(0)}=0,\\
&\|\textbf{w}^{(0)}\|_{H^{s}(\Omega_t)}\lesssim \eps_0(\overline{T}^*-t)^{C_{\eps,a,\bar{a},k,\nu,\mu}},\quad \|\textbf{b}^{(0)}\|_{H^s(\Omega_t)}\lesssim\eps_0(\overline{T}^*-t)^{C_{\eps,a,\bar{a},k,\nu,\mu}},\\
&\|\textbf{w}^{(0)}\|_{\Ccal_1^{k_0+3}}\lesssim \eps_0<\eps,\quad \|\textbf{b}^{(0)}\|_{\Ccal_1^{k_0+3}}\lesssim \eps_0<\eps,\\
&\|E_1^{(0)}\|_{\Ccal_1^{k_0+3}}\lesssim\eps_0<{\eps\over2},\quad \|E_2^{(0)}\|_{\Ccal_1^{k_0+3}}\lesssim\eps_0<{\eps\over2},
\endaligned
\ee
where the error term
$$
\aligned
&E_1^{(0)}:=\Lcal(\textbf{w}^{(0)},\textbf{b}^{(0)}),\\
&E_2^{(0)}:=\Jcal(\textbf{w}^{(0)},\textbf{b}^{(0)}).
\endaligned
$$

The $m$-th error terms are defined by
\bel{E4-6}
\aligned
\Rcal_1(\textbf{h}^{(m)},\textbf{q}^{(m)})&:=\Lcal(\textbf{w}^{(m-1)}+\textbf{h}^{(m)},\textbf{b}^{(m-1)}+\textbf{q}^{(m)})-\Lcal(\textbf{w}^{(m-1)},\textbf{b}^{(m-1)})\\
&\quad-\Pi_{N_m}\Lcal[(\textbf{w}^{(m-1)},\textbf{b}^{(m-1)})](\textbf{h}^{(m)},\textbf{q}^{(m)}),\\
\Rcal_2(\textbf{h}^{(m)},\textbf{q}^{(m)})&:=\Jcal(\textbf{w}^{(m-1)}+\textbf{h}^{(m)},\textbf{b}^{(m-1)}+\textbf{q}^{(m)})-\Jcal(\textbf{w}^{(m-1)},\textbf{b}^{(m-1)})\\
&\quad-\Pi_{N_m}\Jcal[(\textbf{w}^{(m-1)},\textbf{b}^{(m-1)})](\textbf{h}^{(m)},\textbf{q}^{(m)}),\\
\endaligned
\ee
which are also the nonlinear term in approximation problem (\ref{E4-1}) at $(\textbf{w}^{(m-1)},\textbf{b}^{(m-1)})$. The exact form of nonlinear term (\ref{E4-6}) is very complicated, here we does not write it down.
We carry out the tame estimates.

\begin{lemma}
Let viscosity constant $\nu$ and resistivity constant $\mu$ be sufficient big, constants $a\in(0,{1\over2}]$, $\bar{a},k\in(0,1]$ and $\forall s\in\NN^+$.
Assume that $(\textbf{w}^{(m-1)},\textbf{b}^{(m-1)})\in H^s(\Omega_t)\times H^s(\Omega_t)$. Then for any $t\in(0,\overline{T}^*)$, it holds
\bel{E4-9R1}
\aligned
&\|\Rcal_1(\textbf{h}^{(m)},\textbf{q}^{(m)})\|_{\Ccal_1^s}\lesssim N_m^2\Big(\|\textbf{h}^{(m)}\|^2_{\Ccal_1^s}+\|\textbf{q}^{(m)}\|^2_{\Ccal_1^s}\Big),\\
&\|\Rcal_2(\textbf{h}^{(m)},\textbf{q}^{(m)})\|_{\Ccal_1^s}\lesssim N_m^2\Big(\|\textbf{h}^{(m)}\|^2_{\Ccal_1^s}+\|\textbf{q}^{(m)}\|^2_{\Ccal_1^s}\Big).
\endaligned
\ee
\end{lemma}
\begin{proof}
We notice that the highest order of nonlinear term in (\ref{E4-1}) is $2$, and the highest order of derivatives on $x$ in (\ref{E4-6})
are $1$. By (\ref{AH2}) and (\ref{E4-1}), we use the standard Calder\'on-Zygmund theory and Young's inequality to estimate each term in $\Rcal_1(\textbf{h}^{(m)},\textbf{q}^{(m)})$ and $\Rcal_2(\textbf{h}^{(m)},\textbf{q}^{(m)})$, we obtain
$$
\aligned
&\|\Rcal_1(\textbf{h}^{(m)},\textbf{q}^{(m)})\|_{\Ccal_1^s}\lesssim N_m^2\Big(\|\textbf{h}^{(m)}\|^2_{\Ccal_1^s}+\|\textbf{q}^{(m)}\|^2_{\Ccal_1^s}\Big),\\
&\|\Rcal_2(\textbf{h}^{(m)},\textbf{q}^{(m)})\|_{\Ccal_1^s}\lesssim N_m^2\Big(\|\textbf{h}^{(m)}\|^2_{\Ccal_1^s}+\|\textbf{q}^{(m)}\|^2_{\Ccal_1^s}\Big).
\endaligned
$$

\end{proof}

The following Lemma is to show how to construct the $m$-th approximation solution.

\begin{lemma}
Let viscosity constant $\nu$ and resistivity constant $\mu$ be sufficient big, constants $a\in(0,{1\over2}]$, $\bar{a},k\in(0,1]$ and $\forall s\in\NN^+$.
Assume that $(\textbf{w}^{(m-1)},\textbf{b}^{(m-1)})\in\Bcal_{\eps}$.
The linear problem
$$
\aligned
&\Pi_{N_m}\Lcal[(\textbf{w}^{(m-1)},\textbf{b}^{(m-1)})](\textbf{h}^{(m)},\textbf{q}^{(m)})=E_1^{(m-1)},\\
&\Pi_{N_m}\Jcal[(\textbf{w}^{(m-1)},\textbf{b}^{(m-1)})](\textbf{h}^{(m)},\textbf{q}^{(m)})=E_2^{(m-1)},\\
&\nabla\cdot\textbf{h}^{(m)}=0,\quad \nabla\cdot\textbf{q}^{(m)}=0,
\endaligned
$$
with the initial data
$$
\textbf{h}^{(m)}(0,x)=\textbf{h}^{(m)}_0(x),\quad \textbf{q}^{(m)}(0,x)=\textbf{q}^{(m)}_0(x),
$$
and the boundary conditions
$$
\textbf{h}^{(m)}(t,x)|_{x\in\del\Omega_t}=0,\quad \textbf{q}^{(m)}(t,x)|_{x\in\del\Omega_t}=0,
$$
has a solution $(\textbf{h}^{(m)},\textbf{q}^{(m)})\in H^s(\Omega_t)\times H^s(\Omega_t)$ satisfying
\bel{AAE4-10R1}
\|\textbf{h}^{(m)}\|^2_{H^s}+\|\textbf{q}^{(m)}\|^2_{H^s}\lesssim (\overline{T}^*-t)^{C_{\eps,a,\bar{a},k,\nu,\mu}}\Big(\|\textbf{h}_0^{(m)}\|_{H^s}^2+\|\textbf{q}_0^{(m)}\|_{H^s}^2+ \|E_1^{(m-1)}\|^2_{H^s}+ \|E_2^{(m-1)}\|^2_{H^s}\Big),
\ee
and
\bel{E4-10R1}
\|\textbf{h}^{(m)}\|^2_{\Ccal_1^s}+\|\textbf{q}^{(m)}\|^2_{\Ccal_1^s}\lesssim\|\textbf{h}_0^{(m)}\|_{\Ccal_1^s}^2+\|\textbf{q}_0^{(m)}\|_{\Ccal_1^s}^2+ \|E_1^{(m-1)}\|^2_{\Ccal_1^s}+ \|E_2^{(m-1)}\|^2_{\Ccal_1^s},
\ee
where $C_{\eps,a,\bar{a},k,\nu,\mu}$ is a positive constant depending on constants $\eps,a,\bar{a},k,\nu,\mu$, and the error term
\bel{E4-10R2}
\aligned
&E_1^{(m-1)}:=\Lcal(\textbf{w}^{(m-1)},\textbf{b}^{(m-1)})=\Rcal_1(\textbf{h}^{(m)},\textbf{q}^{(m)}),\\
&E_2^{(m-1)}:=\Jcal(\textbf{w}^{(m-1)},\textbf{b}^{(m-1)})=\Rcal_2(\textbf{h}^{(m)},\textbf{q}^{(m)}).
\endaligned
\ee
\end{lemma}
\begin{proof}
Assume that $(\textbf{w}^{(0)},\textbf{b}^{(0)})$ satisfies (\ref{E4-5}). The $m-1$-th approximation solution is
$$
\aligned
&\textbf{w}^{(m-1)}=\textbf{w}^{(0)}+\sum_{i=1}^{m-1}\textbf{h}^{(i)},\\
&\textbf{b}^{(m-1)}=\textbf{b}^{(0)}+\sum_{i=1}^{m-1}\textbf{q}^{(i)}.
\endaligned
$$
Then we will find the $m$-th approximation solution $(\textbf{w}^{(m)},\textbf{b}^{(m)})$, which is equivalent to find $(\textbf{h}^{(m)},\textbf{q}^{(m)})$ such that
\bel{E4-7}
\aligned
&\textbf{w}^{(m)}=\textbf{w}^{(m-1)}+\textbf{h}^{(m)},\\
&\textbf{b}^{(m)}=\textbf{b}^{(m-1)}+\textbf{q}^{(m)}.
\endaligned
\ee
Substituting (\ref{E4-7}) into (\ref{E4-1}), there is
$$
\aligned
&\Lcal(\textbf{w}^{(m)},\textbf{b}^{(m)})=\Lcal(\textbf{w}^{(m-1)},\textbf{b}^{(m-1)})+\Pi_{N_m}\Lcal[(\textbf{w}^{(m-1)},\textbf{b}^{(m-1)})](\textbf{h}^{(m)},\textbf{q}^{(m)})+\Rcal_1(\textbf{h}^{(m)},\textbf{q}^{(m)}),\\
&\Jcal(\textbf{w}^{(m)},\textbf{b}^{(m)})=\Jcal(\textbf{w}^{(m-1)},\textbf{b}^{(m-1)})+\Pi_{N_m}\Jcal[(\textbf{w}^{(m-1)},\textbf{b}^{(m-1)})](\textbf{h}^{(m)},\textbf{q}^{(m)})+\Rcal_2(\textbf{h}^{(m)},\textbf{q}^{(m)}).
\endaligned
$$

Set
$$
\aligned
&\Lcal(\textbf{w}^{(m-1)},\textbf{b}^{(m-1)})+\Pi_{N_m}\Lcal[(\textbf{w}^{(m-1)},\textbf{b}^{(m-1)})](\textbf{h}^{(m)},\textbf{q}^{(m)})=0,\\
&\Jcal(\textbf{w}^{(m-1)},\textbf{b}^{(m-1)})+\Pi_{N_m}\Jcal[(\textbf{w}^{(m-1)},\textbf{b}^{(m-1)})](\textbf{h}^{(m)},\textbf{q}^{(m)})=0,
\endaligned
$$
we supplement it with the initial data
$$
\aligned
&\textbf{h}^{(m)}(0,x)=\textbf{h}^{(m)}_{0}(x):=\textbf{v}_0(x)-\overline{\textbf{v}}(0,x)-\sum_{i=1}^{m-1}\textbf{h}^{(i)}(0,x),\\
&\textbf{q}^{(m)}(0,x)=\textbf{q}^{(m)}_{0}(x):=\textbf{H}_0(x)-\overline{\textbf{H}}(0,x)-\sum_{i=1}^{m-1}\textbf{q}^{(i)}(0,x),
\endaligned
$$
and boundary conditions
$$
\textbf{h}^{(m)}(t,x)|_{x\in\del\Omega_t}=0,\quad \textbf{q}^{(m)}(t,x)|_{x\in\del\Omega_t}=0.
$$

By Proposition 3.1, above problem admits a solution $(\textbf{h}^{(m)},\textbf{q}^{(m)})\in H^s(\Omega_t)\times H^s(\Omega_t)$ with $\nabla\cdot\textbf{h}^{(m)}=0$ and $\nabla\cdot\textbf{q}^{(m)}=0$. Furthermore, by (\ref{E3-25R1X})-(\ref{E3-25R1}), it satisfies
$$
\|\textbf{h}^{(m)}\|^2_{H^s}+\|\textbf{q}^{(m)}\|^2_{H^s}\lesssim (\overline{T}^*-t)^{C_{\eps,a,\bar{a},k,\nu,\mu}}\Big(\|\textbf{h}_0^{(m)}\|_{H^s}^2+\|\textbf{q}_0^{(m)}\|_{H^s}^2+ \|E_1^{(m-1)}\|^2_{H^s}+ \|E_2^{(m-1)}\|^2_{H^s}\Big),
$$
and
$$
\|\textbf{h}^{(m)}\|^2_{\Ccal_1^s}+\|\textbf{q}^{(m)}\|^2_{\Ccal_1^s}\lesssim\|\textbf{h}_0^{(m)}\|_{\Ccal_1^s}^2+\|\textbf{q}_0^{(m)}\|_{\Ccal_1^s}^2+ \|E_1^{(m-1)}\|^2_{\Ccal_1^s}+ \|E_2^{(m-1)}\|^2_{\Ccal_1^s},
$$
where one can see the $m-1$-th error term $E^{(m-1)}$ such that
$$
\aligned
&E_1^{(m-1)}:=\Lcal(\textbf{w}^{(m-1)},\textbf{b}^{(m-1)})=\Rcal_1(\textbf{h}^{(m)},\textbf{q}^{(m)}),\\
&E_2^{(m-1)}:=\Jcal(\textbf{w}^{(m-1)},\textbf{b}^{(m-1)})=\Rcal_2(\textbf{h}^{(m)},\textbf{q}^{(m)}).
\endaligned
$$
\end{proof}

\subsection{Convergence of the approximation scheme}

For any fixed integer $s\geq2$, let $1<\bar{k}<k_0\leq k\leq s$ and
$$
\aligned
&k_m:=\bar{k}+\frac{k-\bar{k}}{2^m},\\
&\alpha_{m+1}:=k_m-k_{m+1}=\frac{k-\bar{k}}{2^{m+1}},
\endaligned
$$
which gives that
\bel{EX1-1}
k_0>k_1>\ldots>k_m>k_{m+1}>\ldots.
\ee

\begin{proposition}
Let viscosity constant $\nu$ and resistivity constant $\mu$ be sufficient big, constants $a\in(0,{1\over2}]$, $\bar{a},k\in(0,1]$, a fixed integer $s\geq2$, $0<k_0\leq s$ and $0<\eps\ll1$. 
The nonlinear equations
\bel{Ex1-1}
\aligned
&\textbf{w}_t+\textbf{w}\cdot\nabla\overline{\textbf{v}}_{\overline{T}^*}+\overline{\textbf{v}}_{\overline{T}^*}\cdot\nabla\textbf{w}+\textbf{w}\cdot\nabla\textbf{w}=\nabla p+\nu\triangle\textbf{w}
+\overline{\textbf{H}}_{\overline{T}^*}\cdot\nabla\textbf{b}+\textbf{b}\cdot\nabla\overline{\textbf{H}}_{\overline{T}^*}\\
&\quad\quad\quad\quad\quad\quad\quad\quad\quad\quad\quad\quad\quad\quad\quad\quad-\nabla(\overline{\textbf{H}}_{\overline{T}^*}\cdot\textbf{b})
+\textbf{b}\cdot\nabla\textbf{b}-\nabla({|\textbf{b}|^2\over 2}),\\
&\textbf{b}_t-\mu\triangle\textbf{b}=\overline{\textbf{H}}_{\overline{T}^*}\cdot\nabla\textbf{w}+\textbf{b}\cdot\nabla\overline{\textbf{v}}_{\overline{T}^*}
-\overline{\textbf{v}}_{\overline{T}^*}\cdot\nabla\textbf{b}-\textbf{w}\cdot\nabla\overline{\textbf{H}}_{\overline{T}^*}
+\textbf{b}\cdot\nabla\textbf{w}-\textbf{w}\cdot\nabla\textbf{b},\\
&\nabla\cdot\textbf{w}=0,\quad \nabla\cdot\textbf{b}=0,
\endaligned
\ee
with small initial data
$$
\textbf{w}(0,x)=\textbf{w}_0(x),\quad \textbf{b}(0,x)=\textbf{b}_0(x),
$$
and boundary conditions
$$
\textbf{w}(t,x)|_{x\in\del\Omega_t}=0,\quad \textbf{b}(t,x)|_{x\in\del\Omega_t}=0,
$$
admits a local solution
$$
\aligned
\textbf{w}^{(\infty)}(t,x)&=\textbf{w}^{(0)}(t,x)+\sum_{m=1}^{\infty}\textbf{h}^{(m)}(t,x)+[{1\over\overline{T}^*}(\overline{T}^*-t)]^{C_{\eps,a,\bar{a},k,\nu,\mu}}\textbf{w}_0(x),\\
\textbf{b}^{(\infty)}(t,x)&=\textbf{b}^{(0)}(t,x)+\sum_{m=1}^{\infty}\textbf{q}^{(m)}(t,x)+[{1\over\overline{T}^*}(\overline{T}^*-t)]^{C_{\eps,a,\bar{a},k,\nu,\mu}}\textbf{b}_0(x),
\endaligned
$$
where $ (t,x)\in(0,\overline{T}^*)\times\Omega_t$, 
$\sum\limits_{m=1}^{\infty}\textbf{h}^{(m)}(t,x)\in \Ccal_1^{k_0}(\Omega_t)$, $\sum\limits_{m=1}^{\infty}\textbf{q}^{(m)}(t,x)\in \Ccal_1^{k_0}(\Omega_t)$,
and $C_{\eps,a,\bar{a},k,\nu,\mu}$ is a positive constant depending on constants $\eps,a,\bar{a},k,\nu,\mu$.

Moreover, it holds
$$
\aligned
&\|\textbf{w}^{(\infty)}\|_{H^s(\Omega_t)}\lesssim (\overline{T}^*-t)^{C_{\eps,a,\bar{a},k,\nu,\mu}},\\
&\|\textbf{b}^{(\infty)}\|_{H^s(\Omega_t)}\lesssim (\overline{T}^*-t)^{C_{\eps,a,\bar{a},k,\nu,\mu}}.
\endaligned
$$

\end{proposition}
\begin{proof}

The proof is based on the induction.
For convenience, we first deal with the case of zero initial data, i.e. $\textbf{w}(0,x)=0$ and $\textbf{b}(0,x)=0$. After that, we discuss the small initial data case.
Note that $N_m=N_0^m$ with $N_0>1$. $\forall m=1,2,\ldots$, we claim that there exists a sufficient small positive constant $\eps$ such that
\bel{E4-12}
\aligned
&\|\textbf{h}^{(m)}\|_{\Ccal_1^{k_m}}+\|\textbf{q}^{(m)}\|_{\Ccal_1^{k_m}}<\eps^{2^m},\\
&\|E^{(m-1)}_1\|_{\Ccal_1^{k_m}}<\eps^{2^{m+1}},\quad \|E^{(m-1)}_2\|_{\Ccal_1^{k_m}}<\eps^{2^{m+1}},\\
&(\textbf{w}^{(m)},\textbf{b}^{(m)})\in\Bcal_{\eps}.
\endaligned
\ee

For the case of $m=1$, we recall that the assumption (\ref{E4-5}) on $(\textbf{w}^{(0)},\textbf{b}^{(0)})$, i.e.
$$
\aligned
&\textbf{w}^{(0)}\neq(0,0,0)^T,\quad \textbf{b}^{(0)}\neq(0,0,0)^T,\\
&\nabla\cdot\textbf{w}^{(0)}=0,\quad \nabla\cdot\textbf{b}^{(0)}=0,\\
&\|\textbf{w}^{(0)}\|_{H^{s}(\Omega_t)}\lesssim \eps_0(\overline{T}^*-t)^{C_{\eps,a,\bar{a},k,\nu,\mu}},\quad \|\textbf{b}^{(0)}\|_{H^s(\Omega_t)}\lesssim\eps_0(\overline{T}^*-t)^{C_{\eps,a,\bar{a},k,\nu,\mu}},\\
&\|\textbf{w}^{(0)}\|_{\Ccal_1^{k_0+3}}\lesssim \eps_0<\eps,\quad \|\textbf{b}^{(0)}\|_{\Ccal_1^{k_0+3}}\lesssim \eps_0<\eps,\\
&\|E_1^{(0)}\|_{\Ccal_1^{k_0+3}}\lesssim\eps_0<{\eps\over 2},\quad \|E_2^{(0)}\|_{\Ccal_1^{k_0+3}}\lesssim\eps_0<{\eps\over 2}.
\endaligned
$$

Note that $\textbf{h}^{(m)}(0,x)=0$ and $\textbf{q}^{(m)}(0,x)=0$.
By (\ref{E4-10R1}), let $0<\eps_0<N_0^{-8}\eps^2<{\eps\over2}\ll1$, we have
$$
\|\textbf{h}^{(1)}\|_{\Ccal_1^{k_1}}+\|\textbf{q}^{(1)}\|_{\Ccal_1^{k_1}}\lesssim \|E_1^{(0)}\|_{\Ccal_1^{k_0}}+\|E_2^{(0)}\|_{\Ccal_1^{k_0}}\lesssim 2\eps_0<\eps.
$$
Moreover, by (\ref{E4-9R1}) and (\ref{E4-10R2}), we derive
$$
\aligned
&\|E_1^{(1)}\|_{\Ccal_1^{k_1}}\lesssim\|\Rcal_1(\textbf{h}^{1)},\textbf{q}^{(1)})\|_{\Ccal_1^{k_1}}\lesssim N_1^2\Big(\|\textbf{h}^{(1)}\|^2_{\Ccal_1^{k_1}}+\|\textbf{q}^{(1)}\|^2_{\Ccal_1^{k_1}}\Big)\lesssim 2\eps_0 N_1^2<\eps^2,\\
&\|E_2^{(1)}\|_{\Ccal_1^{k_1}}\lesssim\|\Rcal_2(\textbf{h}^{1)},\textbf{q}^{(1)})\|_{\Ccal_1^{k_1}}\lesssim N_1^2\Big(\|\textbf{h}^{(1)}\|^2_{\Ccal_1^{k_1}}+\|\textbf{q}^{(1)}\|^2_{\Ccal_1^{k_1}}\Big)\lesssim 2\eps_0 N_1^2<\eps^2,
\endaligned
$$
and
$$
\aligned
\|\textbf{w}^{(1)}\|_{\Ccal_1^{k_1+3}}+\|\textbf{b}^{(1)}\|_{\Ccal_1^{k_1+3}}&\lesssim\|\textbf{w}^{(0)}\|_{\Ccal_1^{k_1+3}}+\|\textbf{b}^{(0)}\|_{\Ccal_1^{k_1+3}}+\|\textbf{h}^{(1)}\|_{\Ccal_1^{k_1+3}}+\|\textbf{q}^{(1)}\|_{\Ccal_1^{k_1+3}}\\
&\lesssim\|\textbf{w}^{(0)}\|_{\Ccal_1^{k_0+3}}+\|\textbf{b}^{(0)}\|_{\Ccal_1^{k_0+3}}+\|E_1^{(0)}\|_{\Ccal_1^{k_0}}+\|E_2^{(0)}\|_{\Ccal_1^{k_0}}\\
&\lesssim \eps,
\endaligned
$$
which means that $(\textbf{w}^{(1)},\textbf{b}^{(1)})\in\Bcal_{\eps}$.

Assume that the case of $m-1$ holds, i.e.
\bel{E4-13}
\aligned
&\|\textbf{h}^{(m-1)}\|_{\Ccal_1^{k_{m-1}}}+ \|\textbf{q}^{(m-1)}\|_{\Ccal_1^{k_{m-1}}}<\eps^{2^{m-1}},\\
&\|E_1^{(m-1)}\|_{\Ccal_1^{k_{m-1}}}<\eps^{2^{m}},\quad \|E_2^{(m-1)}\|_{\Ccal_1^{k_{m-1}}}<\eps^{2^{m}},\\
&(\textbf{w}^{(m-1)},\textbf{b}^{(m-1)})\in\Bcal_{\eps},
\endaligned
\ee
then we prove the case of $m$ holds. Using (\ref{E4-10R1}) and (\ref{E4-13}), we have
\bel{E4-14R1}
\aligned
\|\textbf{h}^{(m)}\|_{\Ccal_1^{k_m}}+\|\textbf{q}^{(m)}\|_{\Ccal_1^{k_m}}&\lesssim \|E_1^{(m-1)}\|_{\Ccal_1^{k_{m}}}+\|E_2^{(m-1)}\|_{\Ccal_1^{k_{m}}}\\
&< \|E_1^{(m-1)}\|_{\Ccal_1^{k_{m-1}}}+\|E_2^{(m-1)}\|_{\Ccal_1^{k_{m-1}}}\\
&<\eps^{2^{m}},
\endaligned
\ee
which combining with (\ref{E4-9R1}), (\ref{E4-10R2}) and (\ref{EX1-1}), it holds
\bel{E4-14}
\aligned
\|E_1^{(m)}\|_{\Ccal_1^{k_{m}}}+\|E_2^{(m)}\|_{\Ccal_1^{k_{m}}}&=\|\Rcal_1(\textbf{h}^{m)},\textbf{q}^{(m)})\|_{\Ccal_1^{k_{m}}}+\|\Rcal_2(\textbf{h}^{m)},\textbf{q}^{(m)})\|_{\Ccal_1^{k_{m}}}\\
&\lesssim N_{m-1}^2\Big(\|E_1^{(m-1)}\|^2_{\Ccal_1^{k_{m-1}}}+|E_2^{(m-1)}\|^2_{\Ccal_1^{k_{m-1}}}\Big)\\
&\lesssim N_{m-1}^2\Big(\|E_1^{(m-1)}\|_{\Ccal_1^{k_{m-1}}}+|E_2^{(m-1)}\|_{\Ccal_1^{k_{m-1}}}\Big)^2\\
&\lesssim N_0^{2(m-1)+4(m-2)}\Big(\|E_1^{(m-2)}\|_{\Ccal_1^{k_{m-2}}}+\|E_2^{(m-2)}\|_{\Ccal_1^{k_{m-2}}}\Big)^{2^2}\\
&\lesssim \ldots,\\
&\lesssim \Big[N_0^4\Big(\|E^{(0)}_1\|_{\Ccal_1^{k_0}}+\|E^{(0)}_2\|_{\Ccal_1^{k_0}}\Big)\Big]^{2^m}.
\endaligned
\ee

So by (\ref{E4-5}),  there is a sufficient small positive constant $\eps_0$ such that
$$
0<N_0^4\Big(\|E^{(0)}_1\|_{\Ccal_1^{k_0}}+\|E^{(0)}_2\|_{\Ccal_1^{k_0}}\Big)<2N_0^4\eps_0<\eps^2,
$$
which combining with (\ref{E4-14}) gives that
$$
\|E_1^{(m)}\|_{\Ccal_1^{k_{m}}}+\|E_2^{(m)}\|_{\Ccal_1^{k_{m}}}<\eps^{2^{m+1}}.
$$
On the other hand, note that $N_m=N_0^m$, by  (\ref{AH2}) and (\ref{E4-14R1})-(\ref{E4-13}), it holds
$$
\aligned
\|\textbf{w}^{(m)}\|_{\Ccal_1^{k_{m}}+3}+\|\textbf{b}^{(m)}\|_{\Ccal_1^{k_{m}}+3}&\lesssim \|\textbf{w}^{(m-1)}\|_{\Ccal_1^{k_{m-1}+3}}+\|\textbf{b}^{(m-1)}\|_{\Ccal_1^{k_{m-1}+3}}+\|\textbf{h}^{(m)}\|_{\Ccal_1^{k_{m}+3}}+\|\textbf{q}^{(m)}\|_{\Ccal_1^{k_{m}+3}}\\
&\lesssim \eps+N_m^3\eps^{2^m}\lesssim\eps.
\endaligned
$$
This means that $(\textbf{w}^{(m)},\textbf{b}^{(m)})\in\Bcal_{\eps}$. Hence we conclude that (\ref{E4-12}) holds.

Furthermore, it follows from (\ref{E4-12}) that the error term goes to $0$ as $m\rightarrow\infty$, i.e.
$$
\lim_{m\rightarrow\infty}\Big(\|E_1^{(m)}\|_{\Ccal_1^{k_m}}+\|E_2^{(m)}\|_{\Ccal_1^{k_m}}\Big)=0.
$$

Therefore, equations (\ref{Ex1-1}) with the zero initial data $\textbf{w}(0,x)=0$ and $\textbf{b}(0,x)=0$, and boundary condition $\textbf{w}(t,x)|_{x\in\del\Omega_t}=0$ and $\textbf{b}(t,x)|_{x\in\del\Omega_t}=0$ admits a solution
$$
\aligned
&\textbf{w}^{(\infty)}=\textbf{w}^{(0)}+\sum_{m=1}^{\infty}\textbf{h}^{(m)}\in \Ccal_1^{k_0}(\Omega_t),\\
&\textbf{b}^{(\infty)}=\textbf{b}^{(0)}+\sum_{m=1}^{\infty}\textbf{q}^{(m)}\in \Ccal_1^{k_0}(\Omega_t).
\endaligned
$$

Next we discuss the case of small initial data
$$
\textbf{w}(0,x)=\textbf{w}_0(x),\quad \textbf{b}(0,x)=\textbf{b}_0(x).
$$
where 
$$
\|\textbf{w}_0(x)\|_{H^s(\Omega_t)}<\eps,\quad \|\textbf{b}_0(x)\|_{H^s(\Omega_t)}<\eps.
$$

We introduce an auxiliary function
$$
\aligned
&\overline{\textbf{w}}(t,x)=\textbf{w}(t,x)-[{1\over\overline{T}^*}(\overline{T}^*-t)]^{C_{\eps,a,\bar{a},k,\nu,\mu}} \textbf{w}_0(x),\\
&\overline{\textbf{b}}(t,x)=\textbf{b}(t,x)-[{1\over\overline{T}^*}(\overline{T}^*-t)]^{C_{\eps,a,\bar{a},k,\nu,\mu}}  \textbf{b}_0(x),
\endaligned
$$
then small initial data is reduced into
$$
\overline{\textbf{w}}(0,x)=0,\quad \overline{\textbf{b}}(0,x)=0,
$$
and equations (\ref{Ex1-1}) is transformed into equations of $(\overline{\textbf{w}}(t,x),\overline{\textbf{b}}(t,x))$. Since $\eps$ is sufficient small and $(\textbf{w}_0(x),\textbf{b}_0(x))\in H^s(\Omega_t)\times H^s(\Omega_t)$,
we can follow above iteration scheme to obtain the local existence of $(\overline{\textbf{w}}(t,x),\overline{\textbf{b}}(t,x))$ for $(t,x)\in(0,\overline{T}^*)\times\Omega_t$.
Furthermore, the local solution of equations (\ref{Ex1-1}) with small initial data takes the form 
$$\Big(\overline{\textbf{w}}(t,x)+[{1\over\overline{T}^*}(\overline{T}^*-t)]^{C_{\eps,a,\bar{a},k,\nu,\mu}} \textbf{w}_0(x),\overline{\textbf{b}}(t,x)+[{1\over\overline{T}^*}(\overline{T}^*-t)]^{C_{\eps,a,\bar{a},k,\nu,\mu}}\textbf{b}_0(x)\Big)$$.

Moreover, we can choose the initial approximation function $(\textbf{w}_{\overline{T}^*}^{(0)}(t,x),\textbf{b}_{\overline{T}^*}^{(0)}(t,x))^T$ depending on the parameter $\overline{T}^*$ continuity. Since the initial data depends on 
 the parameter $T^*$ continuity when we solve the linearized system at each iteration step, so 
$(\textbf{h}^{(m)}_{T^*}(t,x),\textbf{q}^{(m)}_{T^*}(t,x))^T$ also depends on the parameter $T^*$ continuity.
By the exact form of solutions which we constructed, it holds
$$
\aligned
\textbf{w}^{(m+1)}(0,x)&=\textbf{w}^{(0)}_{\overline{T}^*}(0,x)+\sum_{i=1}^{m}\textbf{h}_{T^*}^{(i)}(0,x)+[{1\over\overline{T}^*}(\overline{T}^*-t)]^{C_{\eps,a,\bar{a},k,\nu,\mu}}\textbf{w}_0(x),\\
&=\mathcal{R}_{T^*,\overline{T}^*}(0,x)+[{1\over\overline{T}^*}(\overline{T}^*-t)]^{C_{\eps,a,\bar{a},k,\nu,\mu}}\textbf{w}_0(x),
\endaligned
$$
and
$$
\aligned
\textbf{b}^{(m+1)}(0,x)&=\textbf{b}^{(0)}_{\overline{T}^*}(0,x)+\sum_{i=1}^{m}\textbf{q}_{T^*}^{(i)}(0,x)+[{1\over\overline{T}^*}(\overline{T}^*-t)]^{C_{\eps,a,\bar{a},k,\nu,\mu}}\textbf{b}_0(x),\\
&=\overline{\mathcal{R}}_{T^*,\overline{T}^*}(0,x)+[{1\over\overline{T}^*}(\overline{T}^*-t)]^{C_{\eps,a,\bar{a},k,\nu,\mu}}\textbf{b}_0(x),
\endaligned
$$

then there exists a $\overline{T}^*\in[T^*-\delta,T^*+\delta]$ with $0<\delta\ll1$ such that
$$
\aligned
&\|\mathcal{R}_{T^*,\overline{T}^*}(0,x)\|_{\HH^s(\Omega_t)}=\mathcal{O}(\eps),\\
&\|\overline{\mathcal{R}}_{T^*,\overline{T}^*}(0,x)\|_{\HH^s(\Omega_t)}=\mathcal{O}(\eps).
\endaligned
$$
Thus it holds
$$
\aligned
&\textbf{w}^{(m+1)}(0,x)=[{1\over\overline{T}^*}(\overline{T}^*-t)]^{C_{\eps,a,\bar{a},k,\nu,\mu}}\textbf{w}_0(x)+\mathcal{O}(\eps),\\
&\textbf{b}^{(m+1)}(0,x)=[{1\over\overline{T}^*}(\overline{T}^*-t)]^{C_{\eps,a,\bar{a},k,\nu,\mu}}\textbf{b}_0(x)+\mathcal{O}(\eps).
\endaligned
$$

At last, we recall the time-decay of each of approximation step given in (\ref{AAE4-10R1}), so we obtain
$$
\aligned
&\|\textbf{w}^{(\infty)}\|_{H^s(\Omega_t)}\lesssim (\overline{T}^*-t)^{C_{\eps,a,\bar{a},k,\nu,\mu}},\\
&\|\textbf{b}^{(\infty)}\|_{H^s(\Omega_t)}\lesssim (\overline{T}^*-t)^{C_{\eps,a,\bar{a},k,\nu,\mu}}.
\endaligned
$$

This completes the proof.

\end{proof}

%================================================================================
%================================================================================
%================================================================================
%================================================================================

\textbf{Acknowledgments.}
The author expresses his sincerely thanks to the BICMR of Peking University and Professor Gang Tian for constant support and encouragement,
The author expresses his sincerely thanks to Prof. J.L. Liu for his useful discussion and suggestion.
This work is supported by NSFC No 11771359.

\end{document}